\documentclass{article}
\usepackage[a4paper,margin=1in]{geometry}  
\usepackage[utf8]{inputenc}
\usepackage{dsfont}

\usepackage{float}
\usepackage[dvipsnames,svgnames,cmyk]{xcolor}     
\usepackage{graphicx}         
\usepackage{subcaption}
\usepackage[outercaption]{sidecap}
\usepackage{tabularx}
\usepackage{multirow}       
\usepackage{booktabs}
\usepackage{enumitem}
\usepackage{url}

\usepackage{amssymb}    
\usepackage{amsthm}     
\usepackage{thmtools}   
\usepackage{mathtools}  

\usepackage{mathrsfs}   
\usepackage{dsfont} 
\usepackage{algorithm}    
\usepackage{algpseudocode}  

\usepackage{pgfplots}
\pgfplotsset{compat=newest}
\usepgfplotslibrary{patchplots}
\usepgfplotslibrary{colormaps}
\usepackage{bm}         
\usepackage{bbold}      
\usepackage{comment}
\newcommand{\R}{\mathbf{R}}    
\newcommand{\N}{\mathbf{N}}    
\renewcommand{\vec}[1]{\boldsymbol{#1}}
\newcommand{\id}{id}

\newtheorem{theorem}{Theorem}

\newtheorem{remark}[theorem]{Remark}%
\newtheorem{lemma}[theorem]{Lemma}%
\newtheorem{definition}[theorem]{Definition}
\newtheorem{assumption}[theorem]{Assumption}
\newtheorem{corollary}[theorem]{Corollary}
\begin{document}

\title{Efficiently parallelizable kernel-based multi-scale algorithm}

\author{Federico Lot and Christian Rieger \\
Department of Mathematics and Computer Science, Philipps-Universität Marburg, \\ 
Hans-Meerwein-Straße 6, Marburg, 35032, Hessia, Germany}
\date{}

\maketitle

\abstract{
    The kernel-based multi-scale method has been proven to be a powerful approximation method for scattered data 
	approximation problems which is computationally superior to conventional kernel-based interpolation techniques.
	The multi-scale method is based of an hierarchy of point clouds and compactly supported radial basis functions, typically Wendland functions.
	There is a rich body of literature concerning the analysis of this method including error estimates. This article addresses the efficient 
	parallelizable implementation of those methods. To this end, we present and analyse a monolithic approach to compute the kernel-based multi-scale approximation. 
}


\section{Introduction}\label{sec:Introduction}
Many approximation problems can be cast in the following form: Let  $\Omega \subset \R^d$  be a sufficiently nice bounded domain and 
$f :\Omega \to \R$ a continuous function.  The task is to reconstruct $f$ from its evaluations on discrete sets of scattered points.
For scattered nodes, radial basis functions are a common tool to compute such reconstructions. If the function values are given exactly interpolation by radial basis functions is often performed.
It is, however, known that the plain kernel-based interpolation leads to densely populated and ill-conditioned linear systems. 

As a remedy for these problems, kernel-based multiscale methods are an already well-established alternative. These methods do not provide an interpolation but a good approximation. 
This idea dates back to \cite{Floater:Iske:1996} from an algorithmic point of view. Later, theoretical results were obtained on bounded domains in \cite{Wendland2010} and \cite{Georgoulis:etal:2013,Kempf:Wendland:2023} for extensions to high-dimensional problems using concepts from sparse grids. 
The idea is not restricted to function reconstruction on bounded domains. There are several applications on spheres, see e.g., \cite{LeGia:etal:2010,LeGia:etal:2017} and even to approximation problems for manifold valued functions, see \cite{Sharon:etal:2023} or rough functions, see \cite{Townsend:Wendland:2012}.
Moreover, the concept of multilevel kernel methods also allows for concepts from signal processing or machine learning such as compression, see e.g. \cite{LeGia:Wendland:2014}. 
Finally multiscale kernel techniques are used to solve partial differential equations, see e.g. 
\cite{Farrell:Wendland:2013,LeGia:etal:2012,Wendland2018}.
We expect that our proposed reformulation can be (in a suitably adapted fashion) applied to all of these problems. 

The computational cost of such schemes is not always feasible in real-world application. However, their structure can exploited to ease the computations, as in \cite{Buttner:Kempf:Wendland:2025}. 

In order to compute the approximation in any of those multiscale methods, the data has to be arranged in a hierarchy of point sets which get denser with increasing level. 
The multiscale method works a as follows: One computes the residual of the approximation based on all the levels up to the current level and approximates this residual on the point set corresponding to the current level with a suitable kernel. This method can be made very effective by scaling the kernel according to the levels.  
However, the method as described above works in a sequential way corresponding to the levels. The approximation on each level can be parallelized if needed but there seems no obvious parallelization strategy with respect to the levels. This is the motivation for this article. We present a reformulation of the kernel-based multiscale method as a block-triangular linear system containing all levels, which allows for efficient and scalable techniques from numerical linear algebra, see Theorem \ref{thm:reformulation} and the numerical results in Section \ref{sec:numerics}. Next, we exploit the specific structure of this block-triangular system in order to derive a factorization into a block triangular matrix and a block lower triangular matrix with normalized diagonals. We show that those factors enable an efficient iterative solution strategy. This iterative methods relies essentially on matrix vector multiplications which can be efficiently parallelized. 
For the block diagonal matrix we will use essentially the conjugate gradients method whereas we use a Jacobi method for the block triangular matrix, see Theorem \ref{thm:jacobi}. Finally, we show theoretically that the block triangular matrix has a certain decay in magnitude of the entries (Lemma \ref{lem:Mbound}). 

This allows to reduce the storage of the matrix by a straight-forward thresholding strategy. This is discussed and analyzed in Theorem \ref{thm:decayerror}.
 We conclude with some numerical experiments. 
\smallskip

The remainder of the article is organized as follows:
In section \ref{sec:notation} we introduce the basic idea and the notation of the kernel-based multiscale method. Moreover, we collect the necessary assumptions on the hierarchy of point sets and the compactly supported kernel. In section \ref{sec:matrixformulation} we provide alternative descriptions of the kernel-based multiscale method which enable the efficient implementation. The reformulation consists of one large triangular system which itself is factored into a blockdiagonal matrix and a block lower triangular matrix with normalized block-diagonal.
In section \ref{sec:thresholding}, we exploit a decay in the block lower triangular matrix in order to achieve a more efficient storage. Moreover, we provide an error analysis for the multiscale kernel methods if the matrix with the reduced storage is used.
In section \ref{sec:numerics} we present numerical results.

\section{Notation and basic setting}\label{sec:notation}
For scattered nodes, radial basis functions are a common tool to compute interpolations. Radial basis functions are described by a function $K_{\Phi}:\R^{d}\times \R^{d} \to \R$ with $K_{\Phi}(\vec{x},\vec{y}) =\Phi(\vec{x}-\vec{y})= \phi(\|\vec{x}-\vec{y} \|_2)$, where $\Phi:\R^{d} \to\R$ is a multivariate translation-invariant function and $\phi: \R \to \R$ is a univariate function. We are mostly interested in strictly positive definite functions with compact support.
We collect the assumptions on the kernel for later reference.
\begin{assumption}
	\label{ass:kernel}
	Let $\Phi:\R^{d} \to\R$ be a radial function and $\tau>\frac{d}{2}$.  Let $c_{\Phi},C_{\Phi}>0$ be positive constants, e assume that the Fourier transform of $\Phi$ behaves like 
	\begin{equation} \label{eq:kernelfourier}
      c_{\Phi} \left(1+\|\vec{\omega}\|_{2}^{2}\right)^{-\tau} \le \widehat{\Phi}(\vec{\omega}) \le  C_{\Phi} \left(1+\|\vec{ \omega}\|_{2}^{2}\right)^{-\tau}  
\end{equation}
	and that $\Phi$ has compact support in the unit ball, i.e., $\text{supp}(\Phi) \subset B_{1}(0)=\{ \vec{x} \in \R^d : \| \vec{x}\|_2\le 1\}\subset \R^{d}$. The scaled kernel is defined as
	\begin{equation}\label{eq:kernelscaling}
    \Phi_{\delta}(\vec{x}) := \delta^{-d}\Phi(\vec{x}/\delta).
\end{equation}
\end{assumption}
The Fourier-transform of $\Phi_{\delta}$ is given as
\begin{equation}
     c_{\Phi} \left(1+\delta^{2}\|\bm \omega\|_{2}^{2}\right)^{-\tau} \le \widehat{\Phi_{\delta}}(\bm \omega) \le  C_{\Phi} \left(1+\delta^{2}\|\bm \omega\|_{2}^{2}\right)^{-\tau}.
    \label{eq:deltafourier}
\end{equation}
Note that the condition \eqref{eq:kernelfourier} ensures that $\Phi$ is the reproducing kernel of an Hilbert space $\mathcal{H}_{\Phi}$ which is isometric to the Sobolev space $H^{\tau}(\R^d)$ and the norms are equivalent. 
For multi-scale methods one usually considers not a fixed set of nodes but instead a sequence of node sets.

As usual, for a given set $X \subset \Omega$, we have the fill distance
\begin{equation*}
   h_{X, \Omega} := \sup_{\vec{x} \in \Omega} \min_{\vec{x}_{j} \in X} \|\vec{x} - \vec{x}_{j}\|_{2}, 
\end{equation*}
and the separation distance
\begin{equation*}
     q_{X} := \frac{1}{2} \min_{j \neq k} \| \vec{x}_{j} - \vec{x}_{k}\|_{2}.
\end{equation*}
We collect the assumptions on the point sets for later reference.
\begin{assumption}
    \label{ass:pointset}
    We assume $\Omega$ to be a bounded domain with Lipschitz boundary.
    Let $N:\N=\{0,1,\dots\} \to \N$ denote a sequence of natural numbers and $X_{N(\ell)} = \{\vec{x}^{(\ell)}_{1}, \dots, \vec{x}^{(\ell)}_{N(\ell)}\} \subset \Omega$ be a set of scattered nodes for $\ell\in \N$. We will consider levels up to level $L\in \N$ and define the total number of points as $N=\sum_{\ell=1}^{L}N(\ell)$.
    Moreover, we define $h_{\ell}:= h_{X_{N(\ell)}, \Omega}$ and $q_{\ell}:=h_{X_{N(\ell)}}$.
    We assume that there exist $ \mu \in (0, 1)$, $h_{0} \in \R_{\ge 0}$ and $c_{h} \in (0,1]$ such that
    	\begin{equation*}
       c_{h} \mu h_{\ell} \le h_{\ell+1} \le \mu h_{\ell}, \quad \text{for } \ell \in \N.
    \end{equation*}
    Moreover we require that the sequence of the data sets $X_{N(\ell)}$ is quasi-uniform, i.e. there exists a constant $c_{q}$ such that, for $q_{\ell} = q_{X_{N(\ell)}}$,
    \begin{equation}\label{eq:quasiuniform}
        q_{\ell} \le h_{\ell} \le c_{q} q_{\ell},  \quad \quad \text{for } \ell \in \N.
    \end{equation}
    We further assume that the number of points at a given level is connected to the parameter $\mu$, i.e. that exists positive $c_{\#}$ such that
    \begin{equation}\label{eq:Nmu}
        c_{\#} \mu^{-d\ell} \le N(\ell)
    \end{equation}
    holds. Moreover, we choose 
    \begin{equation}\label{eq:deltadef}
    	\delta_{\ell} := \nu h_{\ell} \quad \text{ for  }\nu \in \left[ \frac{1}{h_{1}} ,\frac{\gamma}{\mu} \right].
    \end{equation}
    for the scaling parameter in \eqref{eq:kernelscaling} on the  point set $X_{N(\ell)}$.
    Finally, we need two additional assumptions, i.e, $C_1 \mu <1$, where $C_1=C_1(\gamma)\ge 1$ is defined in Theorem \ref{thm:errorbound} and
    \begin{equation}
        \label{ass:eq:chmu}
        \mu^{-d}>2.
    \end{equation}
\end{assumption}
We point out that in practical application the condition \ref{ass:eq:chmu} bis  often fulfilled. Indeed, for $d \ge 2$, chosen $\mu \le 0.7$ implies the inequality, additionally, higher values, i.e. $\mu \approx 0.9$, still implies $\mu^{-d}>1.2$ which impact our analysis up to a small constant. Higher values for $\mu$ are less appealing for practical applications. 

Moreover, given \eqref{eq:quasiuniform} there is a matching lower bound for the inequality \eqref{eq:Nmu}. This lower bound can be shown by defining
\begin{equation*}
    \mathcal{X} = \{\vec{y} \in \Omega \, : \, \|\vec{y}-\vec{x}\|_{2} \le h_{1}, \, \forall \vec{x} \in \cup_{\ell=1}^{L} X_{N(\ell)}\} \subset \Omega
\end{equation*}
and noticing that the following inequality 
\begin{equation*}
    \text{Vol}(\mathcal{X}) \le N(1)\text{Vol}(B_{h_{1}}(0)) = V_{d}N(1)h_{1}^{d}
\end{equation*}
holds. Then, with a volume comparison argument we obtain
\begin{equation*}
    N(\ell) \le \frac{\text{Vol($\mathcal{X}$)}}{\text{Vol}(B_{q_{\ell}}(0))} \le N(1) \left(\frac{h_{1}}{q_{\ell}}\right)^{d} \le N(1)c_{q}^{d} \mu^{-d(\ell-1)},
\end{equation*}
or, with the notation $N(0) := N(1) \mu^{d}$
\begin{equation}\label{eq:Nmu2}
    N(\ell) \le N(0) c_{q}^{d} \mu^{-d\ell},
\end{equation}
which is the lower bound for \eqref{eq:Nmu}. 
We point out, that in some results, we do not need quasi-uniformity as given by \eqref{eq:quasiuniform}. 
We will mention explicitly , if this assumption can be dropped. 

Next, we define approximation spaces on the level $1\le \ell \le L$ as
\begin{equation}\label{eq:localspaces}
    W_{\ell} = \text{span}\{\Phi_{\delta_{\ell}}(\cdot, \vec{x}_{j}^{(\ell)}) \, : \, \ 1\le j \le N(\ell) \}\subset \mathcal{H}_{\Phi_{\ell}}.
\end{equation}
Associated with such space are operators
\begin{align*}
    S_{X_{N(\ell)}}&: \mathcal{H}_{\Phi_{\ell}}\to \R^{N(\ell)}, \quad   S_{X_{N(\ell)}}(f) :=\left(f( \vec{x}^{(\ell)}_{1}),\dots , f( \vec{x}^{(\ell)}_{N(\ell)})\right)^{T}
    \\
    S_{X_{N(\ell)}}^{*}&:\R^{N(\ell)} \to W_{\ell}\subset \mathcal{H}_{\Phi_{\ell}}, \quad S_{X_{N(\ell)}}^{*}(c):=\sum_{n = 1}^{N(\ell)} c_{n} \Phi_{\ell}(\cdot - \vec{x}^{(\ell)}_{n})\\
    \mathcal{I}_{\ell}&: \mathcal{H}_{\Phi_{\ell}} \to W_{\ell}\subset \mathcal{H}_{\Phi_{\ell}}, \quad   \mathcal{I}_{\ell}:=
    S_{X_{N(\ell)}}^* (S_{X_{N(\ell)}} S_{X_{N(\ell)}}^*)^{-1} S_{X_{N(\ell)}}
\end{align*}

The global approximation space is given as 
\begin{equation*}
	V_{L} = W_{1} + \dots +W_{L},
\end{equation*}
where the sum is in general not direct.

Assuming we are given the sequence of point sets $X_{N(\ell)}$ for $1\le \ell \le L$, the approximation of a function $f$ works as follows. Initially, we set $f_0 = 0$ and $e_0 = f$. 
For $1\le \ell \le L$ we iterate
\begin{itemize}
	\item compute approximand $s_{\ell} (f):= \mathcal{I}_{\ell}(e_{\ell-1}) \in W_{\ell}$ based on point set $X_{N(\ell)}$
	\item update $f_{\ell} = f_{\ell-1} + s_{\ell}$ and  $e_{\ell} = e_{\ell-1} - s_{\ell}$ 
\end{itemize} 	
to obtain the final approximation $\sum_{\ell=1}^{L}s_{\ell} (f) \in V_{L}$ with $\sum_{\ell=1}^{L}s_{\ell} (f) \approx f$.
From the algorithm, we directly obtain the recursive representation
\begin{equation}\label{eq:recursive:localinterpol}
	 s_{\ell}(f) = \mathcal{I}_{\ell}(e_{\ell-1}) = \mathcal{I}_{\ell}\left(f - \sum_{j = 1}^{\ell-1} s_{j}(f)\right), \quad 1\le \ell \le L.
\end{equation} 
We need special subsets of multi-indices in the sequel.
\begin{definition}
We define for $n,k\in \N$ the set of ordered multi-indices
\begin{equation*}
	\Gamma_{n,k}:=\left\{\vec{\ell}=(\ell_{1}, \ell_2, \dots, \ell_{k}) \in \N^{k} :  \ell_{i} > \ell_{i+1} \text{ for all } 1 \le i \le k \text{ and } n \ge \ell_{1} > \ell_{k} \ge 1\right\}\subset \N^{k}.
\end{equation*}
For given natural numbers $k, j, m \in \N$ with $k >j+1$ and $m\ge 2$,  we define the sets
\begin{equation}\label{eq:Pkjm}
P_{kj}^{(m)} = \left\{ (i_{1}, ..., i_{m}) \in \N^{m} \, | \, k = i_{1} > ... > i_{m} = j\right\} \subset  \N^{m}
\end{equation}
and the set of all the strictly decreasing vectors of natural numbers from $k$ to $j$
\begin{equation}
 P_{kj} := \bigcup\nolimits_{m=2}\nolimits^{k-j+1} P_{kj}^{(m)}.
\label{eq:Pkj}
\end{equation}
\end{definition}
and the associated concatenation of interpolation operators
\begin{equation*}
\mathcal{I}_{\vec{p}} = \mathcal{I}_{p_{1}} \circ \mathcal{I}_{p_2} \cdots \circ \mathcal{I}_{p_{k}},\quad \vec{p}\in \Gamma_{n,k}.
\end{equation*}
We can state a recursive representation of the local interpolation operators.
\begin{theorem}\label{thm:gamma_operator_extended}
Let $\Omega \subset \R^{d}$ and let $X_{N(\ell)}$ for $1\le \ell \le L$ denote a sequence of discrete point sets in $\Omega$. 
We obtain for all $f \in C(\Omega)$
\begin{equation}
    \label{eq:gamma_operator_extended_{1}}
    s_{1}(f) = \mathcal{I}_{1}(f)
\end{equation}
and
\begin{equation}
    \label{eq:gamma_operator_extended}
    s_{\ell}(f) = \mathcal{I}_{\ell} (\id + \sum_{k = 1}^{\ell-1} \sum_{\bm p \in \Gamma_{\ell-1, k}} (-1)^k \, \mathcal{I}_{\vec{p}})(f), \quad \text{for }2\le \ell \le L.
\end{equation}
\end{theorem}
\begin{proof}
The proof is by induction over $\ell$.
The case $\ell=1$ is obvious. Now we assume that the statement was true for all $k \le \ell$ and consider the case $\ell+1$. We compute by \eqref{eq:recursive:localinterpol}
\begin{align}
\label{eq:gamma_operator_proof}
s_{\ell+1}(f) = & \, \mathcal{I}_{\ell+1} (f - \sum_{j = 1}^{\ell} s_{j}(f)) \nonumber \\
        = & \, \mathcal{I}_{\ell+1} (f - \sum_{j=2}^{\ell} \mathcal{I}_{j} (\id + \sum_{k = 1}^{j-1} \sum_{\vec{p} \in \Gamma_{j-1, k}} (-1)^k \, 
       \mathcal{I}_{\vec{p}}(f)) - \mathcal{I}_{1}(f)) \nonumber \\
        = & \, \mathcal{I}_{\ell+1} (\id - \sum_{j=2}^{\ell}\mathcal{I}_{j} (\id + \sum_{k = 1}^{j-1} \sum_{\vec{p} \in \Gamma_{j-1, k}} (-1)^k \, 
        \mathcal{I}_{\vec{p}}) -\mathcal{I}_{1})(f) \nonumber \\
        = & \, \mathcal{I}_{\ell+1} (\id - \sum_{j=2}^{\ell} \mathcal{I}_{j} - \sum_{j=2}^{\ell} \mathcal{I}_{j} \sum_{k = 1}^{j-1} \sum_{\vec{p} \in \Gamma_{j-1, k}} (-1)^k \, 
        \mathcal{I}_{\vec{p}} - \mathcal{I}_{1})(f).
\end{align}
Next, we use the identity
\begin{equation*}
    \sum_{j = 2}^{\ell} \mathcal{I}_{j} \sum_{k = 1}^{j-1} \sum_{\vec{p} \in \Gamma_{j-1, k}} (-1)^k \, \mathcal{I}_{\vec{p}} = 
    \sum_{k = 2}^{\ell} \sum_{\vec{p} \in \Gamma_{\ell, k}} (-1)^{k-1} \, \mathcal{I}_{\vec{p}},
\end{equation*}
where on the rhs we have alternated sum of concatenated operators with length up to $\ell$, on the lhs inside the sum we apply the operator $\mathcal{I}_{j}$ to an object of the same form as the rhs but with length up to $j-1$.  
and notice that $\sum_{\vec{p} \in \Gamma_{\ell, 1}}\mathcal{I}_{\vec{p}} = \sum_{i =1}^{\ell} \mathcal{I}_{i}$.
Finally applying the observations in \eqref{eq:gamma_operator_proof}, gives:
\begin{align*}
s_{\ell+1} = & \, \mathcal{I}_{\ell+1}\left(\id - \sum_{j=2}^{\ell} \mathcal{I}_{j} - \sum_{j=2}^{\ell} \mathcal{I}_{j} \sum_{k = 1}^{j-1} \sum_{\vec{p} \in \Gamma_{j-1, k}} (-1)^{k-1} \, 
\mathcal{I}_{\vec{p}}  -\mathcal{I}_{1}\right) \\
        = & \, \mathcal{I}_{\ell+1} \left(\id - \sum_{j=1}^{\ell} \mathcal{I}_{j} + \sum_{k = 2}^{\ell} \sum_{\vec{p} \in \Gamma_{\ell, k}} (-1)^{k} \, \mathcal{I}_{\vec{p}}\right) \\
        = & \, \mathcal{I}_{\ell+1} \left(\id + \sum_{k = 1}^{\ell} \sum_{\vec{p} \in \Gamma_{\ell, k}} (-1)^{k} \, \mathcal{I}_{\vec{p}}\right) .
\end{align*}
This finishes the proof.
\end{proof}
From \cite[Theorem 1]{Wendland2010}, we have an error bound of the following form:  
\begin{theorem}\label{thm:errorbound}
    Let $f \in H^{\tau}(\Omega)$. Under Assumptions \ref{ass:kernel} and \ref{ass:pointset}, there exists a constant $C_{1}=C_{1}(\gamma)$ such that 
    \begin{equation}\label{eq:errorbound}
    	\left\| f - \sum_{\ell=1}^{L}s_{\ell} (f) \right\|_{L_{2}(\Omega)} \le C_{2}\left(C_{1}  \mu^{\tau} \right)^{L} \|f\|_{H^{\tau}(\Omega)}
    \end{equation}
    for all $L\in \N$.
\end{theorem}
This motivates the assumption $C_1 \mu^{\tau}<1$ in Assumption \ref{ass:pointset}.
Moreover, we can highlight the dependence on the number of points
\begin{corollary}\label{cor:errorbound}
    Suppose that Assumption \ref{ass:pointset} and of Theorem \ref{thm:errorbound} hold. Then, we have the estimate
    \begin{equation}
        \left\| f - \sum_{\ell=1}^{L}s_{\ell} (f) \right\|_{L_{2}(\Omega)} \le C_{2} C_{1}^{L} c_{q}^{\tau} \left(\frac{N(1)}{N}\right)^{\frac{\tau}{d}} \|f\|_{H^{\tau}(\Omega)},
    \end{equation}
    where we recall $N=\sum_{\ell=1}^{L}N(\ell)$.
\end{corollary}
\begin{proof}
    We recall $N(0)=N(1) \mu^{d}$. We get
    \begin{align*}
        N &= \sum_{\ell=1}^{L} N(\ell) \le N(0)c_{q}^{d} \sum_{\ell=1}^{L} \mu^{-d\ell} \le N(0)c_{q}^{d} \frac{1-\mu^{-d(L+1)} -(1-\mu^{-d})}{1-\mu^{-d}}=  N(0)c_{q}^{d} \mu^{-d} \frac{1-\mu^{-dL} }{1-\mu^{-d}}\\
       &\le c_{q}^{d}N(1)\frac{\mu^{-dL} -1}{\mu^{-d}-1}\le c_{q}^{d}N(1)\frac{\mu^{-dL} }{\mu^{-d}-1}\le c_{q}^{d}N(1)\mu^{-dL},
    \end{align*}
    as Assumption \ref{ass:pointset} and in particular \eqref{ass:eq:chmu} when $\mu^{-d}>2$ holds. 
    Then, it follows
    \begin{equation*}
        \mu^{\tau L} \le \left(\frac{N(1)c_{q}^{d} }{N}\right)^{\frac{\tau}{d}} = c_{q}^{\tau} \left(\frac{N(1)}{N}\right)^{\frac{\tau}{d}}.
    \end{equation*}
    On the other hand, for $\mu^{-d}<2$ we get
    \begin{equation*}
        N \le c_{q}^{d} N(1) L \mu^{-dL},
    \end{equation*}
    and therefore
    \begin{equation*}
                \mu^{\tau L} \le c_{q}^{\tau} \left(\frac{L N(1)}{N}\right)^{\frac{\tau}{d}}.
    \end{equation*}
    Inserting this into \eqref{eq:errorbound} yields the assertion. 
\end{proof}
In our settings, we have $2\tau = d+2k+1$, which implies that $\tau  > d/2$ for every choice of Wendland function. Furthermore, Corollary \ref{cor:errorbound} highlights how the approximation converges as the number of points increase and that the rate of convergence depends on the smoothness of the RBF.

\section{Matrix formulation}\label{sec:matrixformulation}
In this section, we fix bases in the spaces $W_{\ell}$, thus we can represent each $s_{\ell}$ via coefficients with respect to that basis.
The coefficients can be obtained by solving a linear system. This will now be discussed in more detail.

We represent the operators $S_{X_{N(\ell_2)}} S_{X_{N(\ell_1)}}^{*}:\R^{N(\ell_1)} \to \R^{N(\ell_2)}$ by matrices, i.e., we have
\begin{equation*}
	S_{X_{N(\ell_1)}} S_{X_{N(\ell_2)}}^{*} \cong  B_{\ell_2,\ell_1}:= \left(\Phi_{\ell_1}(\vec{x}^{(\ell_2)}_{n_2} -\vec{x}^{(\ell_1)}_{n_1}) \right)_
	{\genfrac{}{}{0pt}{}{1\le n_1 \le N(\ell_1)}{1\le n_2 \le N(\ell_2)}} \in \R^{N(\ell_2) \times N(\ell_1)}
\end{equation*}
We will use the short-hand notation $A_{\ell}=B_{\ell, \ell}$, $1\le \ell \le L$ and note that the $A_{\ell}$ are symmetric positive definite matrices as the kernel is (strictly) positive definite.
We obtain the identity
\begin{equation}
    s_{\ell} (f)= \sum_{n=1}^{N(\ell)} \alpha_{n}^{(\ell)} \Phi_{\ell}(\cdot, \vec{x}_{n}^{(\ell)}), 
\end{equation}
where the coefficients $\vec{\alpha}^{(\ell)} \in \R^{N(\ell)}$ satisfy 
the system of linear equations
\begin{equation}
    A_{\ell} \vec{\alpha}^{(\ell)} = \mathbf{f}^{(\ell)} - \sum_{k=1}^{\ell-1} B_{\ell k} \vec{\alpha}^{(k)}, \quad \mathbf{f}^{(\ell)}: = f|_{X_{\ell}} \in \R^{N(\ell)}, 
     \label{eq:mas}
\end{equation}
which is the matrix-analogue of \eqref{eq:recursive:localinterpol}.
Collecting the equations above allow us to represent the whole process as the solution of a linear system.

Thus, we obtain for the approximation in Theorem \ref{thm:errorbound} the representation
\begin{equation} \label{eq:fapproximation}
	 f_{L} = \sum_{\ell=1}^{L}s_{\ell} (f) =  \sum_{\ell=1}^{L}  \sum_{n=1}^{N(\ell)} \alpha_{n}^{(\ell)} \Phi_{\ell}(\cdot, \vec{x}_{n}^{(\ell)}),
\end{equation}
where all coefficients can be obtained by solving the following block-lower-triangular linear system 
\begin{equation}
    \left(\begin{array}{ccccc}
    A_{1} & & & & \\
    B_{21} \, & A_{2} \hfill  & & & \\
    B_{31} \, & B_{32} \, & A_{3} \hfill & & \\
    \vdots \, & \vdots \, & \cdots \, & \ddots & \\
    B_{L 1} \, & B_{L2} \, & B_{L3} \, & \cdots \, & A_{L} \\
\end{array}\right) \, \left(\begin{array}{c}
    \bm\alpha^{(1)} \\
    \bm\alpha^{(2)} \\
    \bm\alpha^{(3)} \\
    \vdots \\
    \bm\alpha^{(L)}
\end{array}\right) = \left(\begin{array}{c}
    \mathbf{f}^{(1)} \\
    \mathbf{f}^{(2)} \\
    \mathbf{f}^{(3)} \\
    \vdots \\
    \mathbf{f}^{(L)}
\end{array}\right).
    \label{eq:bigt}
\end{equation}
We consider a factorisation of the lower triangular matrix into 
\begin{align}
    T_{L} & = 
    \left(\begin{array}{ccccc}
    A_{1} & & & & \\
    B_{21} \, & A_{2} \hfill  & & & \\
    B_{31} \, & B_{32} \, & A_{3} \hfill & & \\
    \vdots \, & \vdots \, & \cdots \, & \ddots & \\
    B_{L1} \, & B_{L2} \, & B_{L3} \, & \cdots \, & A_{L} \\
    \end{array}\right) =\nonumber \\
    & 
    \underbrace{
    \left(\begin{array}{ccccc}
    \mathbb{1}_{N(1)} & & & & \\
    B_{21} A_{1}^{-1} \, & \mathbb{1}_{N(2)}  & & & \\
    B_{31}A_{1}^{-1} \, & B_{32}A_{2}^{-1} \, & \mathbb{1}_{N(3)} & & \\
    \vdots \, & \vdots \, & \cdots \, & \ddots & \\
    B_{L1} A_{1}^{-1} \, & B_{L2} A_{2}^{-1} \, & B_{L3} A_{3}^{-1} \, & \cdots \, & \mathbb{1}_{N(L)} \\
    \end{array}\right)
    }_{=:T^{\prime}_{L}}
    \underbrace{
    \left(\begin{array}{ccccc}
    A_{1} & & & & \\
    & A_{2} \hfill  & & & \\
    & \, & A_{3} \hfill & & \\
    & \, & \, & \ddots & \\
    & \, & \, & \, & A_{L} \\
    \end{array}\right)
    }_{=:D_{L}},
    \label{eq:matrix_decomposition}
\end{align}
where $B_{kj} A_{j}^{-1} \in \R^{N(k)}\times\R^{N(j)}$ and $T^{\prime}_{L}, \, D_{L} \in \R^{N}\times\R^{N}$ with $N:=\sum_{\ell=1}^{L} N(\ell)$. 
We define the cardinal functions 
\begin{align}\label{eq:cardinal}
\chi^{(\ell)}_{n} \in W_{\ell} \quad \text{by the Lagrange conditions }\chi^{(\ell)}_{n}(\vec{x}^{(\ell)}_{m}) = \delta_{n,m} \text{ for all } 1\le n,m\le N(\ell).
\end{align} 
We observe the identity
\begin{equation*}
A_{\ell} 
\left(\begin{array}{c}
    \chi_{1}^{(\ell)}(\vec{x}) \\
    \chi_{2}^{(\ell)}(\vec{x}) \\
    \vdots \\
    \chi_{N(\ell)}^{(\ell)}(\vec{x})
\end{array}\right) = 
\left(\begin{array}{c}
    \Phi_{\ell}(\vec{x}, \vec{x}^{(\ell)}_{1}) \\
    \Phi_{\ell}(\vec{x}, \vec{x}^{(\ell)}_{2}) \\
    \vdots \\
    \Phi_{\ell}(\vec{x}, \vec{x}^{(\ell)}_{N(\ell)})
\end{array}\right) 
\end{equation*}
from which it follows that
\begin{equation}
    \chi_{i}^{(\ell)}(\vec{x}) = \sum\nolimits_{k=1}\nolimits^{N(\ell)} A^{-1}_{ik} \Phi_{\ell}(\vec{x}, \vec{x}^{(\ell)}_{k}) = \sum\nolimits_{k: \|\vec{x}-\vec{x}^{(\ell)}_{k}\|_{2} \le \delta_{\ell}} 
    A^{-1}_{ik} \Phi_{\ell}(\vec{x}, \vec{x}^{(\ell)}_{k}), \quad \text{for all }\vec{x} \in \R^{d}
    \label{eq:chi}
\end{equation}
holds. Thus, we get 
\begin{equation*}
    \left(B_{k\ell}A_{\ell}^{-1}\right)_{j, i} = \sum\nolimits_{h=1}\nolimits^{N(\ell)} A^{-1}_{ih} \Phi_{\ell}(\vec{x}^{(k)}_{j}, \vec{x}^{(\ell)}_{h}) = \chi^{\ell}_{i}(\vec{x}^{(k)}_{\ell}),  \quad 1 \le j \le N(k), \, 1 \le i \le N(\ell).
\end{equation*}

In order to analyse the system further, we need an estimate on the condition number of the matrices which we need to invert.
We obtain
\begin{lemma}\label{lem:condA}
Let Assumptions \ref{ass:kernel} and \ref{ass:pointset} be satisfied and let $A_\ell$ be as in \eqref{eq:mas}, then
\begin{equation}
    \kappa_{2}(A_{\ell}) \le 4^{d} (C_{d}c_{\Phi})^{-1} \left(1 +4 M_{d}^{2} \nu^{2} c_{q}^{2}\right)^{\tau},
    \label{eq:condA}
\end{equation}
with constants $M_{d} = 12\left(\frac{\pi \Gamma^{2}(\frac{d}{2}+1)}{9}\right)^{\frac{1}{d+1}} $ and $C_{d} = \frac{1}{2\Gamma(\frac{d}{2}+1)} \left( \frac{M_d}{2^{\frac{3}{2}}} \right)^d$,
i.e. the condition number is bounded by a constant independent of $\ell$.
\end{lemma}
\begin{proof}
Following the arguments from \cite{Wendland2010}
we obtain the bounds 
\begin{align}
    \|A_{\ell}^{-1}\|_{2\to 2}  &\le q_{\ell}^{d} (C_{d}c_{\Phi})^{-1} \left(1 +4 M_{d}^{2} \delta^{2}_{\ell}/q_{\ell}^{2}\right)^{\tau}\label{eq:Aelllower}\\
    \|A_{\ell}\|_{2\to 2} &\le \delta_{\ell}^{-d} \left( \frac{4 \delta_{\ell}}{q_{\ell}} \right)^{d} = \left( \frac{4}{q_{\ell}} \right)^{d}\label{eq:upper}.
\end{align}
Thus, we obtain for the (relative) condition number
\begin{align}
    \label{eq:condAnonquasiuniform}
    \kappa_{2}(A_{\ell}) & = \|A_{\ell}\|_{2\to 2} \|A_{\ell}^{-1}\|_{2\to 2} \le  4^{d} (C_{d}c_{\Phi})^{-1} \left(1 +4 M_{d}^{2} \delta^{2}_{\ell}/q_{\ell}^{2}\right)^{\tau}.
\end{align}
Note that \eqref{eq:condAnonquasiuniform} holds also non quasi-uniform point sets. Quasi-uniformity yields \eqref{eq:condA} and hence completes the proof.
\end{proof}
We obtain the following bound, which is essentially \cite[Theorem 2.3]{demarchi:wendland:2020} but with constants made explicit to our multilevel setting.
\begin{lemma}\label{lem:lagrangedecay}
Let Assumptions \ref{ass:kernel} and \ref{ass:pointset} be satisfied. Then, there is $\theta>0$ such that for all $1 \le i,j \le N(\ell))$ and all $1\le \ell \le L$, the bound
\begin{equation}
	\label{eq:offdiagonaldecay}
    \left|\left(A_{\ell}^{-1}\right)_{i, j}\right| \le 2(C_{d}c_{\Phi})^{-1} \left(1 +4 M_{d}^{2} \nu^{2} c_{q}^{2}\right)^{\tau} e^{2 \theta \sqrt{d}} q_{\ell}^{d} e^{-\theta \|\vec{x}^{(\ell)}_{i}-\vec{x}^{(\ell)}_{j}\|_{2}/q_{\ell}}
\end{equation}
holds.
Additionally, we have for the Lagrange basis defined in \eqref{eq:chi}
\begin{equation}
    \left| \chi_{i}^{(\ell)}(\vec{x})\right| \le 2(C_{d}c_{\Phi})^{-1} \left(1 +4 M_{d}^{2} \nu^{2} c_{q}^{2}\right)^{\tau} e^{2 \theta \sqrt{d}} e^{c_{q} \nu \theta} (1+\nu c_{q})^{d}e^{-\theta\|\vec{x}^{(\ell)}_{i}-\vec{x}\|_{2}/q_{\ell}}, \quad \vec{x} \in \R^{d},
    \label{eq:chi_bound}
\end{equation}
where the constants $C_{d},M_d$ stem from Lemma \ref{lem:condA}.
\end{lemma}
\begin{proof}
From \cite[Theorem 2.3]{demarchi:wendland:2020} and also \cite[Theorem 3.1]{baxter:etal:1994}
we get for all $1\le i,j \le N(\ell)$
\begin{equation}\label{eq:demarchi}
	\left| (A^{-1}_{\ell})_{i,j} \right|  \le 2 \|A^{-1}_{\ell}\|_{2\to 2} \mu^{\| \vec{y}_{i} -\vec{y}_{j}\|_2},
\end{equation}
with $R:= \sqrt{d}\max\{c_{q}\nu, 4\}$ we have
\begin{align*}
\mu:= \left(\frac{\sqrt{\kappa_{2}(A_{\ell})}-1}{\sqrt{\kappa_{2}(A_{\ell})}+1}\right)^{\frac{1}{R}} \quad \text{ and } \quad 
\vec{y}_{j}:=\left(\lfloor  \frac{\sqrt{d}}{2 q_{\ell}} x^{(\ell)}_{j}[1]\rfloor, \dots , \lfloor  \frac{\sqrt{d}}{2 q_{\ell}} x^{(\ell)}_{j}[d]\rfloor \right)^{T}
\end{align*}
for $x^{(\ell)}_{j}=\left(x^{(\ell)}_{j}[1],\dots,x^{(\ell)}_{j}[d] \right)^{T}$.
Using \eqref{eq:Aelllower}, \eqref{eq:quasiuniform}, \eqref{eq:deltadef} and \eqref{eq:demarchi} yields
\begin{align*}
   \left| (A^{-1}_{\ell})_{i,j} \right|  \le 2(C_{d}c_{\Phi})^{-1} \left(1 +4 M_{d}^{2} \nu^{2} c_{q}^{2}\right)^{\tau} q_{\ell}^{d} e^{-\theta_{\ell}\|\vec{x}_{i}-\vec{x}_{j}\|_{2}/q_{\ell}} e^{2\theta_{\ell}\sqrt{d}},
\end{align*}
where
\begin{equation*}
    \theta_{\ell}:= -\frac{1}{2R} \log{\left(\frac{\sqrt{\kappa_{2}(A_{\ell})}-1}{\sqrt{\kappa_{2}(A_{\ell})}+1}\right)} \quad \ge 0.
\end{equation*}
This is well-defined as $\kappa_{2}(A_{\ell}) >1$ thus $0<\frac{\sqrt{\kappa_{2}(A_{\ell})}-1}{\sqrt{\kappa_{2}(A_{\ell})}+1}<1$.
We point out that 
\begin{equation*}
	t: (1,\infty) \to \R, \quad x \mapsto \log \frac{\sqrt{x}-1}{\sqrt{x}+1}
\end{equation*} 
is increasing as composition of the increasing functions $\log x $, $s(x) =(x-1)/(x+1)$ with $s^{\prime}(x)=2(x+1)^{-2}>0$ and $\sqrt x$. 
Thus, we get a uniform upper bound
\begin{equation}
	\label{eq:theta}
		\theta:= -\frac{1}{2R} \log{\left(\frac{\sqrt{z }-1}{\sqrt{z}+1}\right)}, \quad z:=4^{d} (C_{d}c)^{-1} \left(1 +4 M_{d}^{2} \nu^{2} c_{q}^{2}\right)^{\tau}
\end{equation}
using Lemma \ref{lem:condA}. It then follows as in \cite{demarchi:wendland:2020} that
\begin{align*}
    \left|\chi_{i}^{(\ell)}(\vec{x})\right| & \le 2\|A_{\ell}^{-1}\|_{2} e^{2\theta_{\ell}\sqrt{d}} \|\Phi_{\delta}\|_{L^{\infty}(\R^{d})} e^{c_{q} \nu \theta_{\ell}} (1+\nu c_{q})^{d} e^{-\theta_{\ell}\|\vec{x}^{(\ell)}_{i}-\vec{x}\|_{2}/q_{\ell}}\\
    &\le  2 (C_{d}c_{\Phi})^{-1} \left(1 +4 M_{d}^{2} \nu^{2} c_{q}^{2}\right)^{\tau} e^{2 \theta \sqrt{d}} e^{c_{q} \nu \theta} (1+\nu c_{q})^{d}e^{-\theta\|\vec{x}^{(\ell)}_{i}-\vec{x}\|_{2}/q_{\ell}}.
\end{align*}
\end{proof}

We use the notation
\begin{equation}\label{eq:mathfrakXkell}
\mathfrak{X}_{k, \ell} := B_{k, \ell} A_{\ell}^{-1} = (\chi^{(\ell)}_{i}(\vec{x}^{(k)}_{j}))_{i = 1, ..., N(\ell); \, j = 1, ..., N(k)}, 
\end{equation}
from which we obtain the following representation for the matrix factor in \eqref{eq:matrix_decomposition}
\begin{equation}
T^{\prime}_{L} = 
\left(\begin{array}{ccccc}
    \id_{N(1)} & & & & \\
    \mathfrak{X}_{21} \, & \id_{N(2)}  & & & \\
    \mathfrak{X}_{31} \, & \mathfrak{X}_{32} \, &\id_{N(3)} & & \\
    \vdots \, & \vdots \, & \cdots \, & \ddots & \\
    \mathfrak{X}_{L1} \, & \mathfrak{X}_{L2} \, & \mathfrak{X}_{L3} \, & \cdots \, & \id_{N(L)} \\
    \end{array}\right).
    \label{eq:T_n_prime}
\end{equation}
We observe that
\begin{equation*}
	(\id_{N} -T^{\prime}_{L} )^{L} = 
\left(\begin{array}{ccccc}
    0_{N(1)} & & & & \\
    -\mathfrak{X}_{21} \, & 0_{N(2)}  & & & \\
    -\mathfrak{X}_{31} \, & -\mathfrak{X}_{32} \, & 0_{N(3)} & & \\
    \vdots \, & \vdots \, & \cdots \, & \ddots & \\
    -\mathfrak{X}_{L1} \, & -\mathfrak{X}_{L2} \, & -\mathfrak{X}_{L3} \, & \cdots \, &  0_{N(L)} \\
    \end{array}\right)^L=0,
\end{equation*}
i.e., $\id_{N} -T^{\prime}_{L} $ is nilpotent and hence 
\begin{equation}\label{eq:TinvNeumann}
	(T^{\prime}_{L})^{-1}=\sum_{\ell=0}^{L-1}(\id_{N} -T^{\prime}_{L} )^{\ell}.
\end{equation}

We define for a multi-index $\vec{p}=(p_1=k,\dots,p_m=j)^{T} \in P_{kj}^{(m)}$, where the set is defined in \eqref{eq:Pkjm}, with $p_t>p_s$ for $m\ge t>s\ge 1$, the expression
\begin{equation}
\label{eq:mathfrakXp}
\mathfrak{X}_{\vec{p}}:=\prod_{t=1}^{m-1}\mathfrak{X}_{p_{m-t+1},p_{m-t}}
\end{equation}
We can derive an explicit form of $(T^{\prime}_{L})^{-1}$, which resembles Theorem \ref{thm:gamma_operator_extended}.
\begin{theorem}\label{thm:reformulation}
We obtain the explicit formula for the inverse of the matrix factor in \eqref{eq:matrix_decomposition}
\begin{equation}
    \left(T^{\prime}_{L} \right)^{-1}_{kj} = 
    \begin{cases} \sum\nolimits_{\vec{p} \in P_{kj}} (-1)^{|\vec{p}|-1} \mathfrak{X}_{\vec{p}}, &k>j, \\ 
    \vec{0}, & k<j, \\
    \id,& k=j.
    \end{cases},
    \label{eq:explicit_inverse_T}
\end{equation}
using the notation \eqref{eq:mathfrakXp} and the index set from \eqref{eq:Pkj}.
Moreover, for a given $\vec{y}=\left( \vec{y}_1,\dots,\vec{y}_{L} \right)^{T}\in \R^{N} \cong \R^{N(1)} \times \dots \R^{N(L)}$
\begin{equation*}
T^{\prime-1}_{L} \vec{y} = 
\left(\begin{array}{c}
    \vec{y}_{1}\\
    - \mathfrak{X}_{21} \vec{y}_{1} + \vec{y}_{2} \\
    \sum_{\ell = 1}^{2} \sum_{\vec{p} \in P_{3 \ell}} (-1)^{|\vec{p}|} \mathfrak{X}_{\vec{p}} \vec{y}_{\ell} + \vec{y}_{3} \\
    \vdots \\
    \sum_{\ell = 1}^{L-1} \sum_{\vec{p} \in P_{L \ell}} (-1)^{|\vec{p}|} \mathfrak{X}_{\vec{p}} \vec{y}_{\ell} + \vec{y}_{L} \\
    \end{array}\right).
\end{equation*}

\end{theorem}
\begin{proof}
The expressions in the case $k<j$ follow from the fact that lower triangular matrices form a multiplicative group and a direct calculation yields the statement for $k=j$. Hence, we are left with $k>j$.
As $\left(\id_{N} -T^{\prime}_{L} \right)^{0} = \id$, we have
\begin{equation*}
        \left(\id_{N} -T^{\prime}_{L} \right)^{0}_{kj} = 0, \quad \text{for all } k > j,
\end{equation*}
and thus, $t=0$ does not contribute to the sum in \eqref{eq:TinvNeumann}. Moreover, $ \left(\id_{N} -T^{\prime}_{L} \right)^{1}_{kj} = -\mathfrak{X}_{kj}$ for $k>j$, and
\begin{align*}
     \left(\id_{N} -T^{\prime}_{L}\right)^{2}_{kj} = \sum\nolimits_{h = 0}\nolimits^{L}  \left(\id_{N} -T^{\prime}_{L}\right)_{kh} \left(\id_{N} -T^{\prime}_{L}\right)_{hj} = 
     \sum\nolimits_{k > h > j} \mathfrak{X}_{kh} \mathfrak{X}_{hj},
\end{align*}
since $ \left(\id_{N} -T^{\prime}_{L}\right)_{kj} = 0$, for all $k \le j$. Next, we prove by induction over $t$ that
\begin{equation*}
     \left(\id_{N} -T^{\prime}_{L}\right)^{t}_{kj} = (-1)^{t}\sum\nolimits_{k > h_{1} >... > h_{t-1} > j} \mathfrak{X}_{kh_{1}} \mathfrak{X}_{h_{1}h_{2}} \cdots \mathfrak{X}_{h_{t-1}j},
\end{equation*}
from which the statement follows.  
To this end, we  assume that the statement holds for all $2 \le t \le \iota$, and then show that is also true for $t = \iota+1$.
\begin{align*}
 \left(\id_{N} -T^{\prime}_{L}\right)^{\iota+1}_{kj} & = \sum\nolimits_{h_{\iota} = 0}\nolimits^{L}  \left(\id_{N} -T^{\prime}_{L}\right)^{\iota}_{kh_{\iota}}  \left(\id_{N} -T^{\prime}_{L}\right)_{h_{\iota}j} \\
& = (-1)^{\iota} \sum\nolimits_{h_{\iota} = 0}\nolimits^{L} \sum\nolimits_{k > h_{1} >... > h_{t-1} > h_{\iota}} \mathfrak{X}_{kh_{1}} \mathfrak{X}_{h_{1}h_{2}} \cdots \mathfrak{X}_{h_{\iota-1}h_{\iota}} (-1)\mathfrak{X}_{h_{\iota}j} \\
& = (-1)^{\iota+1}\sum\nolimits_{k > h_{1} >... > h_{\iota} > j} \mathfrak{X}_{kh_{1}} \mathfrak{X}_{h_{1}h_{2}} \cdots \mathfrak{X}_{h_{\iota}j}
\end{align*}
which leads, following our notation, to
\begin{equation*}
 \left(\id_{N} -T^{\prime}_{L}\right)^{\iota+1}_{kj} = (-1)^{\iota+1} \sum\nolimits_{k > h_{1} >... > h_{\iota} > j} \mathfrak{X}_{(k, h_{1}, ..., h_{\iota}, j)} = (-1)^{\iota+1} 
 \sum\nolimits_{\vec{p} \in P^{(\iota+2)}_{kj}} \mathfrak{X}_{\vec{p}}.
\end{equation*}
It also follows that $\left(\id_{N} -T^{\prime}_{L}\right)^{\iota+1}_{kj} = 0$ for all $\, \iota \, : \, \iota \ge k-j$.
Finally, we obtain
\begin{align*}
\left( T^{\prime -1}_{L} \right)_{kj} & = \sum\nolimits_{t = 0}\nolimits^{k-j}  \left(\id_{N} -T^{\prime}_{L}\right)^{l}_{kj} = \sum\nolimits_{t = 1}\nolimits^{k-j} 
\sum\nolimits_{\vec{p} \in P^{(t+1)}_{kj}} (-1)^{t} \mathfrak{X}_{\vec{p}} \\
& = \sum\nolimits_{t = 1}\nolimits^{k-j} \sum\nolimits_{\vec{p} \in P^{(l+1)}_{kj}} (-1)^{|\vec{p}|-1} \mathfrak{X}_{\vec{p}} = \sum\nolimits_{\vec{p} \in P_{kj}} (-1)^{|\vec{p}|-1} \mathfrak{X}_{\vec{p}}.
\end{align*}
Thus we can explicitly write down the inverse
\begin{equation}
    T^{\prime-1}_{L} = 
    \begin{bmatrix}
        \id & \vec{0} & & \\
        - \mathfrak{X}_{21} \, & \id  & &  \\
        \sum_{\vec{p} \in P_{31}} (-1)^{|\vec{p}|-1} \mathfrak{X}_{\vec{p}} \, & - \mathfrak{X}_{32} & &  \\
        \vdots \, & \vdots \, &   & \ddots  \\
        \sum_{\vec{p} \in P_{L1}} (-1)^{|\vec{p}|-1} \mathfrak{X}_{\vec{p}} \, & \sum_{\vec{p} \in P_{L2}} (-1)^{|\vec{p}|-1} \mathfrak{X}_{\vec{p}} \, &   & \cdots \, & \id \\
        \end{bmatrix},
    \label{eq:chi_inv_explicit}
\end{equation}
and therefore we obtain
\begin{equation*}
T^{\prime-1}_{L} \bm y = 
\left(\begin{array}{c}
    \bm y_{1}\\
    - \mathfrak{X}_{21} \vec{y}_{1} + \vec{y}_{2} \\
    \sum_{i = 1}^{2} \sum_{\vec{p} \in P_{3 i}} (-1)^{|\vec{p}|} \mathfrak{X}_{\vec{p}} \vec{y}_{i} + \vec{y}_{3} \\
    \vdots \\
    \sum_{i = 1}^{L-1} \sum_{\vec{p} \in P_{L i}} (-1)^{|\vec{p}|} \mathfrak{X}_{\vec{p}} \vec{y}_{i} + \vec{y}_{L} \\
    \end{array}\right)
\end{equation*}
which finishes the proof.
\end{proof}
Though such a representation is theoretically appealing, it is of limited practical use as the index sets will make the sums large. Hence, we need a more efficient way to solve such linear systems.
To this end, we observe that we can split the solution of the linear system into a coupled system via (see \eqref{eq:matrix_decomposition})
\begin{equation}\label{eq:split}
    T_{L} \vec{\alpha} = \mathbf{f} \Longleftrightarrow \left\{ \begin{array}{c} D_{L} \vec{\alpha} = \vec{\beta}\\ 
    T_{L}^{\prime} \vec{\beta} = \mathbf{f} 
    \end{array} 
\right\}.
\end{equation}
The two systems will be treated separately.

\subsection{The block diagonal system}
We first note that the linear system of equations 
\begin{equation}\label{eq:blockdiagonal}
	D_{L} \vec{\alpha} = \vec{\beta}
\end{equation} 
does not allow for a uniform upper bound on the condition number, as we have to take into consideration eigenvalues on different levels $1\le \ell \le L$. It is, however, a block diagonal system where each block $A_{\ell}$ has a uniform bound on 
the condition number as given in Lemma \ref{lem:condA}.

\begin{theorem}\label{thm:cg}
The overall cost for the solution of the linear system $D_{L} \vec{\alpha} =\vec{\beta}$ with the blockwise conjugate gradient methods (see \cite{wendland_2017}) up to an error tolerance 
\begin{equation*}
    \frac{\| \vec{\alpha} - \vec{\alpha}^{\ast} \|_{2}}{\| \vec{\alpha}^{(0)} - \vec{\alpha}^{\ast} \|_{2}}\le \varepsilon 
\end{equation*}
is 
\begin{equation*}
    \operatorname{cost}(CG) \le \mathcal{O} \left(  N  \lceil \frac{1}{2} \sqrt{C_{cg}} \log{\left(\frac{ (C_{d}c_{\Phi})^{-1} \left(1 +4 M_{d}^{2} \nu^{2} c_{q}^{2}\right)^{\tau}) \sqrt{L} q_{1}^{d/2}}{\varepsilon} \right)} \rceil \right).
\end{equation*}
\end{theorem}
\begin{proof}
We observe 
\begin{align}\label{eq:blockdiagonal_levelwise}
	D_{L} \vec{\alpha} = \vec{\beta} \Leftrightarrow A_{\ell} \vec{\alpha}_{\ell} = \vec{\beta}_{\ell}, \quad 1\le \ell \le L,
\end{align}
where $\vec{\alpha}_{\ell},\vec{\beta}_{\ell} \in \R^{N(\ell)}$.
Choosing a starting vector $\vec{\alpha}^{(0)}_{\ell} \in \R^{N(\ell)}$, we get from \cite[Theorem 6.19, p.200]{wendland_2017}
the following error bound for the iterates of the conjugate gradient method applied to the linear system $A_{\ell} \vec{\alpha}_{\ell} = \vec{\beta}_{\ell}$
\begin{equation*}
    \| \vec{\alpha}^{(k)}_{\ell} - \vec{\alpha}^{\ast}_{\ell} \|_{A_{\ell}} \le 2 \left( \frac{\sqrt{\kappa_{2}(A_{\ell})}-1}{\sqrt{\kappa_2(A_{\ell})}+1} \right)^{k} \, \| \vec{\alpha}^{(0)}_{\ell} - \vec{\alpha}^{\ast}_{\ell} \|_{A_{\ell}},
\end{equation*}
where $\vec{\alpha}^{\ast}_{\ell} \in \R^{N(\ell)}$ denotes the true solution.
By the same argument as in the proof of Lemma \ref{lem:lagrangedecay},
we obtain a uniform upper bound 
\begin{equation*}
    C_{cg}:=\max_{1\le \ell \le L} \frac{\sqrt{\kappa_{2}(A_{\ell})}-1}{\sqrt{\kappa_2(A_{\ell})}+1} =\log{\left(\frac{\sqrt{z }-1}{\sqrt{z}+1}\right)}\quad \text{with } 
     z:=4^{d} (C_{d}c)^{-1} \left(1 +4 M_{d}^{2} \nu^{2} c_{q}^{2}\right)^{\tau},
\end{equation*}
which does not depend on the level.
For the iterative solution of the linear system \eqref{eq:blockdiagonal} with the blockwise applied conjugate gradient method, we obtain an error bound
\begin{equation*}
\| \vec{\alpha}_{\ell}^{(k)} - \vec{\alpha}_{\ell}^{\ast} \|_{A_{\ell}} \le \frac{\varepsilon}{\sqrt{L}(C_{d}c_{\Phi})^{-\frac{1}{2}} \left(1 +4 M_{d}^{2} \nu^{2} c_{q}^{2}\right)^{\frac{\tau}{2}}} q^{-\frac{d}{2}}_{\ell} \| \vec{\alpha}_{\ell}^{(0)} - \vec{\alpha}_{\ell}^{\ast} \|_{A_{\ell}}
\end{equation*}
after at most
\begin{equation}\label{eq:keps}
k_{\ell}(\varepsilon) \sim \lceil \frac{1}{2} \sqrt{C_{cg}} \log{\left(\frac{ (C_{d}c_{\Phi})^{-\frac{1}{2}} \left(1 +4 M_{d}^{2} \nu^{2} c_{q}^{2}\right)^{\frac{\tau}{2}} \sqrt{L}}{\varepsilon} q^{\frac{d}{2}}_{\ell}\right)} \rceil
\end{equation}
many iteration steps, where the constants are defined in Lemma \ref{lem:condA} and $q_{\ell} \le \mu^{\ell}$.
Thus, we obtain using Lemma \ref{lem:condA} an estimate
\begin{align*}
	\left\| \vec{\alpha}^{(k)} -\vec{\alpha}^{\ast}\right\|_{2}^{2} &\le \sum_{\ell=1}^{L} \lambda_{\max}(A^{-1}_{\ell})\| \vec{\alpha}_{\ell}^{(k)} - \vec{\alpha}_{\ell}^{\ast} \|^2_{A_{\ell}} \\& 
	\le (C_{d}c_{\Phi})^{-1} \left(1 +4 M_{d}^{2} \nu^{2} c_{q}^{2}\right)^{\tau} \sum_{\ell=1}^{L}   q^d_{\ell}\| \vec{\alpha}_{\ell}^{(k)} - \vec{\alpha}_{\ell}^{\ast} \|^2_{A_{\ell}}
	\le \varepsilon^{2} \left\| \vec{\alpha}^{(0)} -\vec{\alpha}^{\ast}\right\|_{2}^{2}. 
\end{align*}
Moreover, we note that the numerical cost of one iteration in the conjugate gradient method is dominated by the costs of a matrix vector multiplication with the system matrix. 
In the matrix $A_{\ell}$ the number of non-zero entries per row is bounded by 
\begin{equation}\label{eq:rowcost}
	C_{n}:= \# \{k : \|\vec{x}_{k}^{\ell}-\vec{x}_{n}^{\ell}\|_{2} \le \delta_{\ell}\} \le \left(1+\delta_{\ell}/q_{\ell}\right)^{d} \le \left(1+\nu c_{q}\right)^{d}, \quad 1\le n \le N(\ell).
\end{equation}
This implies that 
\begin{equation*}
	\operatorname{cost}(\vec{x}_{\ell} \mapsto A_{\ell} \vec{x}_{\ell}) \in \mathcal{O}(N(\ell)).
\end{equation*}
Thus, we obtain for the overall cost 
\begin{equation*}
	\operatorname{cost}(CG) \le \mathcal{O} \left(  \sum_{\ell=1}^{L}N(\ell) k_{\ell}(\varepsilon) \right),
\end{equation*}
which concludes the proof in view of \eqref{eq:keps}.
\end{proof}
Numerical results about Theorem \ref{thm:cg} can be found in \cite{Wendland2010}.
\subsection{The lower block triangular system}
We begin with the observation, that the linear system (see \eqref{eq:matrix_decomposition})
\begin{align} \label{eq:split_system_triangular}
       T_{L}^{\prime} \vec{\beta} = \mathbf{f}
\end{align}
can be solved exactly with the Jacobi method in at most $L$ steps, where the right-hand-side is $\mathbf{f} =\left(f_{X(1)}, \dots , f_{X(L)}  \right)^{T}\in \R^{N} $ for some $f\in H^{\tau}(\Omega)$.

\begin{theorem}\label{thm:jacobi}
The Jacobi iterative method applied to the system \eqref{eq:split_system_triangular} 
 converges for every starting point $\vec{\beta}_{0}$ to the solution in exactly $L$ steps. 
\end{theorem}

\begin{proof} 
The iterative step of the Jacobi method is:
\begin{equation}\label{eq:jacobi}
    \vec{\beta}_{m+1} = \mathbf{f} +  \left( \id - T_{L}^{\prime}\right)\vec{\beta}_{m}.
\end{equation}
Then we have for arbitrary given starting value $\vec{\beta}_{0}$ that
\begin{equation*}
\vec{\beta}_{m}  =\sum_{t=0}^{m-1}\left( \id - T_{L}^{\prime}\right)^t \mathbf{f} + \left( \id - T_{L}^{\prime}\right)^{m} \vec{\beta}_{0},
\end{equation*}
which can be proven by induction over $m$. As $\left( \id - T_{L}^{\prime}\right)^L=0$, we get for $m=L$
\begin{equation}\label{eq:exactjacobi}
    \vec{\beta}_{L} =\sum_{t=0}^{L-1}\left( \id - T_{L}^{\prime}\right)^t \mathbf{f} + \left( \id - T_{L}^{\prime}\right)^{L} \vec{\beta}_{0} = 
    \sum_{t=0}^{L-1}\left( \id - T_{L}^{\prime}\right)^t \mathbf{f}= (T_{L}^{\prime })^{-1} \mathbf{f},
\end{equation}
which shows the claimed convergence after $L$ steps.
\end{proof}

We now consider the full system \eqref{eq:split} and obtain the following statement:
\begin{theorem}\label{thm:costsforeps}
	Let $\varepsilon \in (0,1)$ be given. The solution of \eqref{eq:split} with the methods outlined in Theorem \ref{thm:cg} and Theorem \ref{thm:jacobi} yields an approximation 
    \begin{equation*}
        \frac{\| \vec{\alpha} - \vec{\alpha}^{\ast} \|_{2}}{\| \vec{\alpha}^{(0)} - \vec{\alpha}^{\ast} \|_{2}}\le \varepsilon 
    \end{equation*}
    at costs of 
	\begin{equation*}
		\operatorname{costs} (\text{solve}_{\eqref{eq:split}}) \in \mathcal{O} \left( LN^2+ N  \lceil \frac{1}{2} \sqrt{C_{cg}} \log{\left(\frac{ (C_{d}c_{\Phi})^{-1} \left(1 +4 M_{d}^{2} \nu^{2} c_{q}^{2}\right)^{\tau}) \sqrt{L} q_{1}^{d/2}}{\varepsilon} \right)} \rceil \right).
	\end{equation*}
\end{theorem}
 \begin{proof}
 	We observe that the Jacobi solution to $ T_{L}^{\prime}\vec{\beta}=\mathbf{f}$ is exact. The numerical cost are $L$ Iterations with one matrix vector multiplication, i.e, 
	\begin{equation*}
		\operatorname{cost}(\text{Jacobi})\in \mathcal{O} ( L N^2).
	\end{equation*}
	The block CG-method is analyzed in Theorem \ref{thm:cg}. Thus, we obtain an error $\varepsilon$ at costs
	\begin{equation*}
		\operatorname{cost}(CG) \le \mathcal{O} \left(  N  \lceil \frac{1}{2} \sqrt{C_{cg}} \log{\left(\frac{ (C_{d}c_{\Phi})^{-1} \left(1 +4 M_{d}^{2} \nu^{2} c_{q}^{2}\right)^{\tau}) \sqrt{L} q_{1}^{d/2}}{\varepsilon} \right)} \rceil \right).
	\end{equation*}
	Hence, the proof follows.
 \end{proof}
The remaining goal is to reduce the $N^2$ term in Theorem \ref{thm:costsforeps} by exploiting decay in the matrices $\mathfrak{X}_{k, \ell} $.
We observe that the matrices can be compressed by simple thresholding.

\section{Thresholding}\label{sec:thresholding}
Here we aim to effectively set to zero smaller entries in the matrix involved in the Jacobi method. However, we first start with some preliminary analysis on the original matrix $\vec{M}:=\id -T_{L}^{\prime}$, see \eqref{eq:T_n_prime}. 
We obtain
\begin{align}\label{eq:M}
	\vec{M} \vec{c} = (\id -T_{L}^{\prime})\vec{c}=
\begin{pmatrix}
    0_{N(1)} & & & \\
    \mathfrak{X}_{21} \, & 0_{N(2)}  &  & \\
    \vdots \,  \, & \cdots \, & \ddots & \\
    \mathfrak{X}_{L1} \, & \mathfrak{X}_{L2}  \, & \cdots \, & 0_{N(L)} 
    \end{pmatrix}
    \begin{pmatrix}
    	\vec{c}^{(1)} \\ \vec{c}^{(2)} \\ \vdots \\ \vec{c}^{(L)}
    \end{pmatrix}
    = 
    \begin{pmatrix}
    	0 \\
	 \mathfrak{X}_{21} \vec{c}^{(2)} \\
	 \vdots\\
	 \sum_{\ell=1}^{L-1}\mathfrak{X}_{L \ell} \vec{c}^{(\ell)} 
    \end{pmatrix},
 \end{align}
where we have due to \eqref{eq:mathfrakXkell}
\begin{align*}
	\mathfrak{X}_{k \ell} \vec{c}^{(\ell)} = \left( \sum_{i=1}^{N(\ell)} \chi^{(\ell)}_{i}(\vec{x}^{(k)}_{1}) \vec{c}^{\ell}_{i} , \dots, \sum_{i=1}^{N(\ell)} \chi^{(\ell)}_{i}(\vec{x}^{(k)}_{N(k)}) \vec{c}^{\ell}_{i} \right)^{T} \in \R^{N(k)}.
\end{align*}

\begin{lemma}\label{lem:Mbound}
 Let Assumptions \ref{ass:kernel} and \ref{ass:pointset} be satisfied. Then, for the matrix $\vec{M}=\id -T_{L}^{\prime}$ in the Jacobi method (see \eqref{eq:jacobi}), we get the bound 
 \begin{equation} \label{eq:Mbound}
	\left\| M \right\|_{2 \to 2} \le \sqrt{\left\| M \right\|_{\infty \to \infty}\left\| M \right\|_{1 \to 1}}\le C C_{\Sigma} \sqrt{2 L c_{q}^{-d}} (c_h)^{-\frac{dL}{2}}\mu^{-\frac{d(L-1)}{2}}
\end{equation}
where the constants $C$ and $C_{\Sigma}$ are defined in the proof.
\end{lemma}
\begin{proof}
	We observe 
	\begin{align*}
		\left\| M \right\|_{1 \to 1}  
		= \max_{1 \le \ell \le L-1} \max_{1\le k \le N(\ell)} \sum_{j=\ell+1}^{L} \sum_{i=1}^{N(j)} \left| \chi^{(\ell)}_{k}( \vec{x}^{(j)}_{i})\right| \\
		\left\| M \right\|_{\infty \to \infty}  
        = \max_{1 \le j \le L} \max_{1\le i \le N(j)} \sum_{\ell=1}^{j-1} \sum_{k=1}^{N(\ell)} \left| \chi^{(\ell)}_{k}( \vec{x}^{(j)}_{i})\right|. 
	\end{align*}
    Thus, we can employ results on decaying Lagrange functions in both cases.
	First we employ the Lagrange function decay from \eqref{eq:chi_bound} and denote 
	\begin{equation*}
		C:=2 (C_{d} c_{\Phi})^{-1}(1+C_{\Phi}c_{q}^{2}\nu^{2})^{\tau} e^{2 \theta \sqrt{d}} e^{c_{q} \nu \theta} (1+\nu c_{q})^{d}.
	\end{equation*} 
	Thus, we obtain 
	\begin{align*}
		\left\| M \right\|_{1 \to 1} \le C \max_{1 \le \ell \le L-1} \max_{1\le k \le N(\ell)} \sum_{j=\ell+1}^{L} \sum_{i=1}^{N(j)} e^{-\theta\|\vec{x}^{(\ell)}_{k}-\vec{x}^{(j)}_{i} \|_{2}/q_{\ell}}
	\end{align*}
    Here the summation is handled with the trick from \cite{Narcowich91}.
	We define for $1\le \ell \le L$ and  $m \in \N_{0}$
	\begin{equation*}
		\mathcal{A}^{(\ell)}_{(\ell;k)}(m):= B_{(m+1) q_{\ell}}(\vec{x}^{(\ell)}_{k})\setminus B_{m q_{\ell}}(\vec{x}^{(\ell)}_{k}), \quad 1\le k \le N(\ell).
	\end{equation*} 
	Then, we obtain as in \cite[Theorem 12.3]{wendland_2004} that the following bound
\begin{equation*}
    \#\left(X_{N(j)} \cap \mathcal{A}^{(\ell)}_{(\ell;k)}(m)  \right) \le (1+(m+1)q_{\ell}/q_{j})^{d}, \quad \text{for all } 1\le k \le N(\ell)
\end{equation*}
holds true. We get
\begin{align*}
    \left\| M \right\|_{1 \to 1} &\le C \max_{1 \le \ell \le L-1} \max_{1\le k \le N(\ell)} \sum_{j=\ell+1}^{L} \sum_{i=1}^{N(j)} e^{-\theta\|\vec{x}^{(\ell)}_{k}-\vec{x}^{(j)}_{i}\|_{2}/q_{\ell}}\\
    &\le C \max_{1 \le \ell \le L-1} \max_{1\le k \le N(\ell)} \sum_{j=\ell+1}^{L} \sum_{m=0}^{\infty} \sum_{\vec{x}^{(j)}_{i}\in X_{N(j)} \cap \mathcal{A}^{(\ell)}_{(\ell;k)}(m) }e^{-\theta\|\vec{x}^{(\ell)}_{k}-\vec{x}^{(j)}_{i}\|_{2}/q_{\ell}}\\
    &\le C \max_{1 \le \ell \le L-1} \max_{1\le k \le N(\ell)} \sum_{j=\ell+1}^{L} \sum_{m=0}^{\infty}(1+(m+1)q_{\ell}/q_{j})^{d} e^{-\theta m}\\
    &\le C \max_{1 \le \ell \le L-1} \sum_{j=\ell+1}^{L}  \left(\frac{q_{\ell}}{q_{j}} \right)^d \sum_{m=0}^{\infty}(m+2)^d e^{-\theta m} \\
\end{align*}
Now, we observe 
    $q_{\ell} \le h_{\ell} \le \mu^{\ell} h_{0} $ and 
    $q_{j}\ge c_{q}h_{j}\ge c_q (c_h \mu)^{j} h_0$. 
Thus, we have 
  $  \frac{q_{\ell}}{q_j} \le \mu^{\ell -j} c^{-1}_q c_h^{-j} $.
This yields
 \begin{align*}
    \left\| M \right\|_{1 \to 1} &\le C  c_{q}^{-d} \max_{1 \le \ell \le L-1} \mu^{d\ell}\sum_{j=\ell+1}^{L}  \left(\mu c_{h}\right)^{-dj} \sum_{m=0}^{\infty}(m+2)^d e^{-\theta m} .
\end{align*}
Additionally we use the bound
    $\sum_{m=0}^{\infty}(m+2)^d e^{-\theta m} =: C_{\Sigma}$
and hence,
\begin{align*}
    \left\| M \right\|_{1\to 1}  &= C C_{\Sigma}  c_{q}^{-d} \max_{1 \le \ell \le L-1} \mu^{d\ell } \sum_{j=\ell+1}^{L} (c^{-d}_h \mu^{-d})^{j}\\
    &= C C_{\Sigma}  c_{q}^{-d} \max_{1 \le \ell \le L-1} \mu^{d\ell } \frac{(c_h\mu)^{-d(\ell+1)} - (c_h\mu)^{-d(L+1)}}{1-(c_h\mu)^{-d}}\\
    &=C C_{\Sigma}  c^{-d}_h \mu^{-d} c_{q}^{-d} \max_{1 \le \ell \le L-1}  \frac{(c_h)^{-dL}\mu^{d(\ell -L)} - c^{-d\ell}_h  }{(c_h\mu)^{-d}-1}\\ 
    &\le 2C C_{\Sigma}  c_{q}^{-d} \max_{1 \le \ell \le L-1}  (c_h)^{-dL}\mu^{d(\ell -L)} - c^{-d\ell}_h \\ 
    &\le 2C C_{\Sigma}  c_{q}^{-d} (c_h)^{-dL}\mu^{-d(L-1)}. 
\end{align*} 
Moreover, we obtain
 \begin{align*}
    \left\| M \right\|_{\infty \to \infty} &\le C \max_{1 \le j \le L} \max_{1\le i \le N(j)} \sum_{\ell=1}^{j-1} \sum_{k=1}^{N(\ell)} e^{-\theta\|\vec{x}^{(\ell)}_{k}-\vec{x}^{(j)}_{i}\|_{2}/q_{\ell}}\\
&\le C \max_{1 \le j \le L} \max_{1\le i \le N(j)} \sum_{\ell=1}^{j-1} \sum_{m=0}^{\infty} \sum_{\vec{x}^{(\ell)}_{k}\in X_{N(\ell)} \cap \mathcal{A}^{(\ell)}_{j;i}(m) }
e^{-\theta\|\vec{x}^{(\ell)}_{k}-\vec{x}^{(j)}_{i}\|_{2}/q_{\ell}}
\end{align*}
Again, we use the notation 
\begin{equation*}
\mathcal{A}^{(\ell)}_{(j;i)}(m):= B_{(m+1) q_{\ell}}(\vec{x}^{(j)}_{i})\setminus B_{m q_{\ell}}(\vec{x}^{(j)}_{i}), \quad 1\le i \le N(j).
\end{equation*}
and obtain 
for $1\le i \le N(j)$ the uniform bound
\begin{equation*}
    \#\left(X_{N(\ell)} \cap \mathcal{A}^{(\ell)}_{(j;i)}(m)  \right) \le (1+(m+1)q_{\ell}/q_{\ell})^{d} \le (2+m)^d.
\end{equation*}
Thus, we obtain
\begin{align*}
    \left\| M \right\|_{\infty \to \infty} &\le 
	C \max_{1 \le j \le L} \max_{1\le i \le N(j)} \sum_{\ell=1}^{j-1} \sum_{m=0}^{\infty} \sum_{\vec{x}^{(\ell)}_{k}\in X_{N(\ell)} \cap \mathcal{A}^{(\ell)}_{j;i}(m) }
	e^{-\theta\|\vec{x}^{(\ell)}_{k}-\vec{x}^{(j)}_{i}\|_{2}/q_{\ell}}\\
	&\le C  \max_{1 \le j \le L} \max_{1\le i \le N(j)} \sum_{\ell=1}^{j-1} \sum_{m=0}^{\infty}(m+2)^d e^{-\theta m} \\
    &\le C C_{\Sigma}  \max_{1 \le j \le L} \max_{1\le i \le N(j)} (j-1) 
    \le C C_{\Sigma}  L.
\end{align*}
Finally, this yields
\begin{equation*}
	\left\| M \right\|_{2 \to 2} \le \sqrt{\left\| M \right\|_{\infty \to \infty}\left\| M \right\|_{1 \to 1}}\le C C_{\Sigma}  \sqrt{2 L c_{q}^{-d}} (c_h)^{-\frac{dL}{2}}\mu^{-\frac{d(L-1)}{2}}
\end{equation*}
which finishes the proof.
\end{proof}
 The estimate on the number of points in the proof was rough.
\begin{remark}
    In the proof we used a bound on the number of points in $\left(X_{N(j)} \cap \mathcal{A}^{(\ell)}_{(\ell;k)}(m) \right)$ of the form $\left(X_{N(j)} \cap  B_{(m+1) q_{\ell}}(\vec{x}^{(\ell)}_{k}) \right)$. A more accurate bound will be:
    \begin{align*}
        \#\left(X_{N(j)} \cap \mathcal{A}^{(\ell)}_{(\ell;k)}(m)  \right) &\le (1+(m+1)q_{\ell}/q_{j})^{d}-(mq_{\ell}/q_{j}-1)^{d}_{+} \\
        & = (2+q_{\ell}/q_{j})\sum_{s=1}^{d} \left((1+(m+1)q_{\ell}/q_{j})^{d-s} +(mq_{\ell}/q_{j}-1)^{s-1}_{+}\right).
    \end{align*}
    which since $\ell < j$, therefore $q_{\ell} > q_{j}$, behave as $\sim d(2+q_{\ell}/q_{j})(m q_{\ell}/q_{j})^{d-1}$.
    However, for our purposes the accuracy gain is not enough to justify the associated loss of readability.
\end{remark}
Now we define perturbed matrices 
\begin{align*}
	 \left(\tilde{\mathfrak{X}}_{\ell j }(T) \right)_{\genfrac{}{}{0pt}{}{1\le i \le N(\ell)}{1\le k \le N(j)}}:=\begin{cases} \chi^{(j)}_{k}(\vec{x}^{(\ell)}_{i}), & \| \vec{x}^{(\ell)}_{k}-\vec{x}^{(j)}_{i}\| < T \max \{ q_{\ell},q
	 _{j}\} =Tq_j \\
	 0, &\| \vec{x}^{(\ell)}_{k}-\vec{x}^{(j)}_{i}\| \ge T \max \{ q_{\ell} ,q_{j}\}=Tq_j \end{cases}
\end{align*}
from which we get the perturbed matrix
\begin{align}\label{eq:perturbedmatrix}
	\tilde{M}(T):=
\begin{pmatrix}
    0_{N(1)} & & & \\
    \tilde{\mathfrak{X}}_{21}(T) \, & 0_{N(2)}  &  & \\
    \vdots \,  \, & \cdots \, & \ddots & \\
    \tilde{\mathfrak{X}}_{L1}(T) \, & \tilde{\mathfrak{X}}_{L2}(T)  \, & \cdots \, & 0_{N(L)} 
    \end{pmatrix}.
\end{align}

We now aim to investigate the perturbed solution of system \eqref{eq:split_system_triangular} and its impact on the coupled system \eqref{eq:split}, namely

\begin{equation}\label{eq:perturbed_split}
    \begin{array}{c} D_{L} \tilde{\vec{\alpha}} = \tilde{\vec{\beta}}\\ 
    (id_{N}-\tilde{M}(T)) \tilde{\vec{\beta}} = \mathbf{f}.
    \end{array}.
\end{equation}

In the following different preliminary results will be achieved. We obtain 

\begin{lemma}\label{lem:Mpertubationbound}
Let Assumptions \ref{ass:kernel} and \ref{ass:pointset} be satisfied.
Let $M$ and $ \tilde{M}(T)$, be as in \eqref{eq:M} and \eqref{eq:perturbedmatrix}, respectively, for a fixed $T>0$. Then, we get
\begin{align*}
	\left\| M- \tilde{M}(T) \right\|_{2 \to 2}\le CC_{\Sigma} e^{\theta} \sqrt{2 L c_{q}^{-d}} (c_h)^{-\frac{dL}{2}}\mu^{-\frac{d(L-1)}{2}} T^{d} e^{-\theta T}, 
\end{align*}
where $\theta$ is given in \eqref{eq:theta}.
\end{lemma}
\begin{proof}
We get by Lemma \ref{lem:condA}
\begin{align*}
	&\left\| M- \tilde{M}(T) \right\|_{1 \to 1}  \le  \max_{1 \le \ell \le L-1} \max_{1\le k \le N(\ell)} \sum_{j=\ell+1}^{L} \sum_{1\le i \le N(j)} \left|  (\mathfrak{X}_{j \ell})_{i k} -(\tilde{\mathfrak{X}}_{j \ell}(T))_{i k} \right| \\
	&\le  \max_{1 \le \ell \le L-1} \max_{1\le k \le N(\ell)} \sum_{j=\ell+1}^{L} \sum_{\genfrac{}{}{0pt}{}{1\le i \le N(j)}{\| \vec{x}^{(\ell)}_{k} -\vec{x}^{(j)}_{i}\|\ge T q_{\ell}  }} 
	 \left| \chi^{(\ell)}_{k}(\vec{x}^{(j)}_{i}) \right| \\
	&\le C \max_{1 \le \ell \le L-1} \max_{1\le k \le N(\ell)} \sum_{j=\ell+1}^{L} \sum_{m=T}^{\infty} \sum_{\vec{x}^{(j)}_{i}\in X_{N(j)} \cap \mathcal{A}^{(\ell)}_{(\ell;k)}(m) }e^{-\theta\|\vec{x}^{(\ell)}_{k}-\vec{x}^{(j)}_{i}\|_{2}/q_{\ell}}\\
	&\le C  \max_{1 \le \ell \le L-1} \sum_{j=\ell+1}^{L} \left(\frac{q_{\ell}}{q_{j}} \right)^d \sum_{m=T}^{\infty}(m+2)^d e^{-\theta m}\\
	& \le 2 C  c_{q}^{-d} (c_h)^{-dL}\mu^{-d(L-1)} \sum_{m=T}^{\infty}(m+2)^d e^{-\theta m}.
\end{align*}
Now, we use with $C_{\Sigma}$ from Lemma \ref{lem:Mbound}
\begin{align*}
	 \sum_{m=T}^{\infty}(m+2)^d e^{-\theta m}&= \sum_{m=0}^{\infty}(m+T+2)^d e^{-\theta (m+T)} \le e^{-\theta T} T^{d} \sum_{m=0}^{\infty}(m+3)^d e^{-\theta m} \le  C_{\Sigma} e^{\theta} T^{d} e^{-\theta T}
\end{align*}
Moreover, we have
\begin{align*}
	\left\| M- \tilde{M}(T) \right\|_{\infty \to \infty}  &= \max_{2 \le j \le L} \max_{1\le i \le N(j)} \sum_{\ell=1}^{j-1} \sum_{k=1}^{N(\ell)} \left| (\mathfrak{X}_{j \ell})_{i k}  - (\tilde
	{\mathfrak{X}}_{j \ell}(T))_{i k} \right| \\
		&= \max_{1 \le j \le L} \max_{1\le i \le N(j)} \sum_{\ell=1}^{j-1} \sum_{\genfrac{}{}{0pt}{}{1\le k \le N(\ell)} {\| \vec{x}^{(\ell)}_{k} -\vec{x}^{(j)}_{i}\|\ge T q_{\ell}  }}\left| \chi^{(\ell)}_{k}( \vec{x}^{(j)}_{i})\right|\\
		& \le C \max_{1 \le j \le L} \max_{1\le i \le N(j)} \sum_{\ell=1}^{j-1} \sum_{m=T}^{\infty} \sum_{\vec{x}^{(\ell)}_{k}\in X_{N(\ell)} \cap \mathcal{A}^{(j)}_{i}(m) }
	e^{-\theta\|\vec{x}^{(\ell)}_{k}-\vec{x}^{(j)}_{i}\|_{2}/q_{\ell}}\\
	&\le C L \sum_{m=T}^{\infty}(m+2)^d e^{-\theta m}\le C C_{\Sigma} e^{\theta} L T^{d} e^{-\theta T}.
\end{align*}
This concludes the proof.
\end{proof}

\begin{lemma}\label{lem:pert1}
Let Assumptions \ref{ass:kernel} and \ref{ass:pointset} be satisfied.
Let $\vec{\beta}_0, \vec{y} \in \R^{N}$ be given, $M, \tilde{M}(T)$ as in Lemma \ref{lem:Mpertubationbound}. We consider 
	\begin{align*}
    		\vec{\beta}_{m+1} &:= \vec{y} + M \vec{\beta}_{m}, \quad m\ge 1 \\
		\tilde{\vec{\beta}}_{m+1}& := \vec{y} + \tilde{M}(T)\tilde{\vec{\beta}}_{m}, \quad \tilde{\vec{\beta}}_0=\vec{\beta}_{0}.
	\end{align*}
	Then it holds that
	\begin{equation*}
		\| \vec{\beta}_{m} -\tilde{\vec{\beta}}_{m}\|_{2} \le \left\|\vec{y} \right\|_{2} \left\|M - \tilde{M}(T)\right\|_{2 \to 2} \frac{1- (m+1)\left\| M \right\|_{2 \to 2}^{m} + m\left\| M \right\|_{2 \to 2}^{m+1} }{(1-\| M \|)^{2}}.
	\end{equation*}
\end{lemma}
\begin{proof}
	We first observe that
	\begin{equation*}
		\vec{\beta}_{m+1}= \sum_{k=0}^{K} M^{k}\vec{y} + M^{K+1}\vec{\beta}_{m-K}, \quad \text{for all } 0\le K\le m.   
	\end{equation*}
	In particular, we get
	\begin{equation*}
		\vec{\beta}_{m+1}=\sum_{k=0}^{m} M^{k}\vec{y} + M^{m+1}\vec{\beta}_{0}
	\end{equation*}
	We obtain $e_0:=\| \vec{\beta}_0 -\tilde{\vec{\beta}}_0\|_{2} = 0$ and some algebra leads us to
	\begin{align*}
		e_{m+1}:=\| \vec{\beta}_{m+1} -\tilde{\vec{\beta}}_{m+1}\|_{2} & \le  \left\| \sum_{k=0}^{m} M^{k}\vec{y} - \sum_{k=0}^{m} \tilde{M}(T)^{k}\vec{y} \right\|_{2}
        \le \left\|\vec{y} \right\|_{2} \left\| \sum_{k=0}^{m} M^{k} - \tilde{M}(T)^{k} \right\|_{2 \to 2} \\
        &\le \left\|\vec{y} \right\|_{2} \sum_{k=1}^{m} \left\|(M - \tilde{M}(T))(\sum_{s=1}^{k} M^{k-s} \tilde{M}(T)^{s-1}) \right\|_{2 \to 2} \\
        &\le \left\|\vec{y} \right\|_{2} \left\|M - \tilde{M}(T)\right\|_{2 \to 2} \sum_{k=1}^{m}\sum_{s=1}^{k} \left\| M^{k-s} \right\|_{2 \to 2} \left\| \tilde{M}(T)^{s-1}\right\|_{2 \to 2}.
	\end{align*}
    Then, since $\| \tilde{M}^{k} \|_{2 \to 2} \le \| M^{k} \|_{2 \to 2}$ for every $k$, we obtain
    \begin{align*}
        \| \vec{\beta}_{m+1} -\tilde{\vec{\beta}}_{m+1}\|_{2} &\le \left\|\vec{y} \right\|_{2} \left\|M - \tilde{M}(T)\right\|_{2 \to 2} \sum_{k=1}^{m} \sum_{s=1}^{k} \left\| M^{k-s} \right\|_{2 \to 2} \left\| M^{s-1}\right\|_{2 \to 2} \\
        &\le \left\|\vec{y} \right\|_{2} \left\|M - \tilde{M}(T)\right\|_{2 \to 2} \sum_{k=1}^{m} k\left\| M \right\|_{2 \to 2}^{k-1} \\
        &\le \left\|\vec{y} \right\|_{2} \left\|M - \tilde{M}(T)\right\|_{2 \to 2} \frac{1- (m+1)\left\| M \right\|_{2 \to 2}^{m} + m\left\| M \right\|_{2 \to 2}^{m+1} }{(1-\| M \|_{2 \to 2})^{2}}.
    \end{align*}
\end{proof}

After having stated a result as sharp as possible, we will take into account

\begin{corollary}\label{cor:pert1}
    The result of Lemma \ref{lem:pert1} can be relaxed under the assumption that $\|M\|_{2 \to 2} > 2$. Indeed,
    \begin{equation*}
        \frac{1- (L+1)\left\| M \right\|_{2 \to 2}^{L} + L\left\| M \right\|_{2 \to 2}^{L+1} }{(1-\| M \|_{2 \to 2})^{2}} \le \frac{L\|M\|_{2 \to 2}^{L+1}}{(\|M\|_{2 \to 2}-1)^{2}} \le 4L\|M\|_{2 \to 2}^{L-1},
    \end{equation*}
 which implies
    \begin{equation*}
        \| \vec{\beta}_{L} -\tilde{\vec{\beta}}_{L}\|_{2} \le \left\|\vec{y} \right\|_{2} \left\|M - \tilde{M}(T)\right\|_{2 \to 2} 4L\|M\|_{2 \to 2}^{L-1}.
    \end{equation*}
    On the other hand, for $\|M\|_{2 \to 2} \le 2$, 
    \begin{equation*}
        \sum_{k=1}^{m} k\left\| M \right\|_{2 \to 2}^{k-1} \le \sum_{k=1}^{m} k2^{k-1} = 1-(L+1)2^{L}+L2^{L+1} = (L-1)2^{L}+1
    \end{equation*}
    holds. Therefore, a more precise representation could be 
    \begin{equation*}
        \| \vec{\beta}_{L} -\tilde{\vec{\beta}}_{L}\|_{2} \le \left\|\vec{y}  \right\|_{2} \left\|M - \tilde{M}(T)\right\|_{2 \to 2} 4 L\max(2, \|M\|_{2 \to 2})^{L-1}.
    \end{equation*}
\end{corollary}

Now, before stating the approximation error of the truncated solution of the Jacobi method, we introduce some assumption, later we will show that, in practical applications, such assumptions are not restrictive.

\begin{assumption}\label{ass:decay}
    We assume that $L \in \N $ is large enough such that 
    \begin{equation}\label{eq:ass1:thmdecayerror}
    	2CC_{\Sigma}L e^{\theta/L} c_{q}^{d/2L-d/2} \le (c_{h}\mu)^{-dL/2},
    \end{equation}
    where $C_{\Sigma}$ stems from Lemma \ref{lem:Mbound}.
    Moreover, we assume $T>0$ large enough such that 
    \begin{equation}\label{eq:ass2:thmdecayerror}
        T^{d} \le  e^{\frac{\theta}{2}T}
    \end{equation}
    with $\theta$ given in \eqref{eq:theta}.
\end{assumption}

Then, finally

\begin{theorem}\label{thm:decayerror}
Suppose that Assumptions \ref{ass:kernel} and  \ref{ass:pointset} hold, and suppose that $L,T$ satisfy \ref{ass:decay}.
    Let $\vec{\beta}$ and $\tilde{\vec{\beta}}$ be the solution of the block lower triangular system \eqref{eq:split_system_triangular} and \eqref{eq:perturbed_split} respectively.
    Then, we have
	\begin{equation*}
		\left\| \vec{\beta}_{L} -\tilde{\vec{\beta}}_{L}\right\|_{2} \le \sqrt{N(1)} e^{-\frac{\theta}{2}T} (c_h \mu)^{-dL^{2}}\|f\|_{L_{\infty}(\Omega)}.
	\end{equation*}
\end{theorem}
\begin{proof}
    Using Lemma \ref{lem:Mbound}, we have 
    \begin{align*}
         \left\|   M \right\|_{2 \to 2} \le CC_{\Sigma} \sqrt{2 L e^{-\theta} c_{q}^{-d}} (c_h)^{-\frac{dL}{2}}\mu^{-\frac{d(L-1)}{2}},
    \end{align*}
    which safely bounds also $\max(2, \|M\|_{2 \to 2})$.

    Using Lemma \ref{lem:pert1}, precisely Corollary \ref{cor:pert1}, and Lemma \ref{lem:Mpertubationbound}, we obtain the bound
    \begin{align*}
        \left\| \vec{\beta}_{L} -\tilde{\vec{\beta}}_{L}\right\|_{2}  &\le \left\|\mathbf{f} \right\|_{2} 4\left\|M - \tilde{M}(T)\right\|_{2 \to 2} L\left\| M \right\|_{2 \to 2}^{L-1} \\
        &\le 4L\left\|\mathbf{f} \right\|_{2} e^{\theta} T^{d} e^{-\theta T} \left(CC_{\Sigma} \sqrt{2 L c_{q}^{-d}} (c_h)^{-\frac{dL}{2}}\mu^{-\frac{d(L-1)}{2}}\right)^{L} \\
        &\le 4L\|f\|_{L_{\infty}(\Omega)} \sqrt{N} e^{\theta} T^{d} e^{-\theta T} \left(CC_{\Sigma} \sqrt{2 L c_{q}^{-d}} (c_h)^{-\frac{dL}{2}}\mu^{-\frac{d(L-1)}{2}}\right)^{L}. 
    \end{align*}
    Thus, \eqref{eq:Nmu2} leads to
    \begin{align*}
        \left\| \vec{\beta}_{L} -\tilde{\vec{\beta}}_{L}\right\|_{2} &\le 4L\|f\|_{L_{\infty}(\Omega)} \sqrt{N(1) c_{q}^{d} \mu^{-dL}} e^{\theta} T^{d} e^{-\theta T} \left(CC_{\Sigma} \sqrt{2 L c_{q}^{-d}} (c_h)^{-\frac{dL}{2}}\mu^{-\frac{d(L-1)}{2}}\right)^{L} \\
        &= 4L\sqrt{N(1)}\|f\|_{L_{\infty}(\Omega)} e^{\theta} T^{d} e^{-\theta T} \left((CC_{\Sigma})^{2} 2 L c_{q}^{d/L-d} c_{h}^{-dL} \mu^{-dL} \right)^{\frac{L}{2}}.
    \end{align*}
    Then, Assumption \ref{ass:decay} implies $T^{d} e^{-\theta T} \le e^{-\frac{\theta}{2}T}$ and
    \begin{equation*}
        4Le^{\theta}\left((CC_{\Sigma})^{2} 2 L c_{q}^{d/L-d} c_{h}^{-dL} \mu^{-dL} \right)^{\frac{L}{2}} \le (c_{h} \mu)^{-dL^{2}},
    \end{equation*}
    which finishes the proof.
\end{proof}

Next, we couple the thresholding error to the CG error.
\subsection{Error coupling}\label{subsec:errorcoupling}
We now turn to function recontruction.
Now we have to consider a perturbation of the block diagonal system \eqref{eq:blockdiagonal}, i.e,
\begin{equation*}
	D_{L} \tilde{\vec{\alpha}} = \tilde{\vec{\beta}} \Longleftrightarrow A_{\ell} \tilde{\vec{\alpha}}_{\ell}=\tilde{\vec{\beta}}_{\ell}, \quad 1\le \ell \le L.
\end{equation*} 
We recall from Equation \eqref{eq:exactjacobi} the relation $\vec{\beta}=(T^{\prime}_{L})^{-1}\vec{f}$.
We obtain the following Lemma:
\begin{lemma}\label{lem:alphabound}
	Suppose that Assumptions \ref{ass:kernel}, \ref{ass:pointset} and \ref{ass:decay} hold. Let $f \in H^{\tau}(\R^d)$ and let $\vec{\alpha} \in \R^{N}$ be the solutions of the \eqref{eq:blockdiagonal} and $\tilde{\vec{\alpha}} \in \R^{N}$ be the solution of the perturbed system. Then, we get the bound
    \begin{equation}\label{eq:alphabound}
        \| \tilde{\vec{\alpha}}-\vec{\alpha} \|_{2} \le q_{1}^{d} \frac{(1+4M_{d}^{2}\nu^{2}c_{q}^{2})^{\tau}}{C_{d}c_{\Phi}} \sqrt{N(1)}e^{-\frac{\theta}{4}T}\|f\|_{L_{\infty}(\Omega)}
    \end{equation}
    holds.
\end{lemma}
\begin{proof}
   The bound
		$\| \tilde{\vec{\alpha}}-\vec{\alpha} \|_2 \le \|D_{L}^{-1}\|_{2\to 2} \|\vec{\beta} -\tilde{\vec{\beta}}\|_2$
    implies
    \begin{align*}
        \| \tilde{\vec{\alpha}}-\vec{\alpha} \|_{2} & \le \max_{1 \le \ell \le L}{\|A_{\ell}^{-1}\|_{2 \to 2}} \|\vec{\beta} -\tilde{\vec{\beta}} \|_{2}  \le q_{1}^{d}(C_{d}c_{\Phi})^{-1}(1+4M_{d}^{2}\nu^{2}c_{q}^{2})^{\tau} \|\vec{\beta} -\tilde{\vec{\beta}} \|_{2},
	\end{align*}
	where we used \eqref{eq:condAnonquasiuniform}.
	Thus, we obtain using Theorem \ref{thm:decayerror}
	\begin{align*}
		\| \tilde{\vec{\alpha}}-\vec{\alpha} \|_{2} \le q_{1}^{d} \frac{(1+4M_{d}^{2}\nu^{2}c_{q}^{2})^{\tau}}{C_{d}c_{\Phi}} \sqrt{N(1)} e^{-\frac{\theta}{2}T} (c_{h}\mu)^{-dL^{2}} \|f\|_{L_{\infty}(\Omega)}.
	\end{align*}
    Using
    \begin{equation} \label{eq:Testimate}
        (c_{h}\mu)^{-dL^{2}} \le e^{\frac{\theta}{4}T} \Leftrightarrow -dL^{2}\ln(c_{h}\mu) \le \frac{\theta}{4}T
    \end{equation}
    we have
    \begin{equation}
        \| \tilde{\vec{\alpha}}-\vec{\alpha} \|_{2} \le q_{1}^{d} \frac{(1+4M_{d}^{2}\nu^{2}c_{q}^{2})^{\tau}}{C_{d}c_{\Phi}} \sqrt{N(1)} e^{-\frac{\theta}{4}T} \|f\|_{L_{\infty}(\Omega)}, \quad \text{for} \quad T \ge \frac{-4dL^{2}}{\theta} \ln(c_{h}\mu).
    \end{equation}
\end{proof} 

Moreover, the following result will prove useful itself.

\begin{corollary}\label{cor:deltaalphabound}
    Suppose that Assumptions  \ref{ass:kernel}, \ref{ass:pointset} and \ref{ass:decay} hold and let $\vec{\alpha} \in \R^{N}$ be the solution of  \eqref{eq:blockdiagonal} and $\tilde{\vec{\alpha}} \in \R^{N}$ be the solution of the perturbed system. Then defining $\vec{\alpha}(\delta)$ such that
    \begin{equation*}
        \alpha_{i}^{\ell}(\delta) = \delta_{\ell}^{-d/2}\alpha_{i}^{\ell},
    \end{equation*}
    we obtain that the following bound
    \begin{equation}\label{eq:deltaalphabound}
        \| \tilde{\vec{\alpha}}(\delta)-\vec{\alpha}(\delta) \|_{2} \le q_{1}^{d/2} \frac{(1+4M_{d}^{2}\nu^{2}c_{q}^{2})^{\tau}}{\nu^{d/2}C_{d}c_{\Phi}} \sqrt{N(1)}e^{-\frac{\theta}{4}T} \|f\|_{L_{\infty}(\Omega)}
    \end{equation}
    holds true.
\end{corollary}
\begin{proof}
    It is sufficient to show that system \eqref{eq:blockdiagonal} is equivalent to
    \begin{equation*}
        D_{L}(\delta) \vec{\alpha}(\delta) = \vec{\beta},
    \end{equation*}
    where $D_{L}(\delta)$ is the diagonal block matrix Diag($\delta_{1}^{d/2} A_{1}$, $\dots$, $\delta_{L}^{d/2} A_{L}$). Then, we have
    \begin{equation*}
		\| \tilde{\vec{\alpha}}(\delta)-\vec{\alpha}(\delta) \| \le \|D_{L}(\delta)^{-1}\| \|\vec{\beta} -\tilde{\vec{\beta}} \|,
    \end{equation*}
    which implies
    \begin{align*}
        \| \tilde{\vec{\alpha}}(\delta)-\vec{\alpha}(\delta) \|_{2} & \le \max_{1 \le \ell \le L}{\delta_{\ell}^{-d/2}\|A_{\ell}^{-1}\|_{2 \to 2}} \|\vec{\beta} -\tilde{\vec{\beta}} \|_{2} \\
        & \le \max_{1 \le \ell \le L}{\delta_{\ell}^{-d/2}q_{\ell}^{d}}(C_{d}c_{\Phi})^{-1}(1+4M_{d}^{2}\nu^{2}c_{q}^{2})^{\tau} \|\vec{\beta} -\tilde{\vec{\beta}} \|_{2} \\
        & \le \max_{1 \le \ell \le L}{q_{\ell}^{d/2}} \nu^{-d/2} (C_{d}c_{\Phi})^{-1}(1+4M_{d}^{2}\nu^{2}c_{q}^{2})^{\tau} \|\vec{\beta} -\tilde{\vec{\beta}} \|_{2} \\
        & \le q_{1}^{d/2} \nu^{-d/2} (C_{d}c_{\Phi})^{-1}(1+4M_{d}^{2}\nu^{2}c_{q}^{2})^{\tau} \|\vec{\beta} -\tilde{\vec{\beta}} \|_{2}.
	\end{align*}
    Lastly, with Theorem \ref{thm:decayerror} we have
    \begin{equation*}
        \| \tilde{\vec{\alpha}}(\delta)-\vec{\alpha}(\delta) \|_{2} \le q_{1}^{d/2} \nu^{-d/2} (C_{d}c_{\Phi})^{-1}(1+4M_{d}^{2}\nu^{2}c_{q}^{2})^{\tau} \sqrt{N(1)}e^{-\frac{\theta}{4}T}\|f\|_{L_{\infty}(\Omega)}
    \end{equation*}
    for $T \ge \frac{-4dL^{2}}{\theta} \ln(c_{h}\mu)$.
\end{proof}

Finally, we can state

\begin{theorem}\label{thm:approximant_difference_with_perturbation}
    Let $f \in H^{\tau}(\Omega)$, $L \in \N$. Under Assumptions \ref{ass:kernel}, \ref{ass:pointset} and \ref{ass:decay}, let $f_{L}$ as in \eqref{eq:fapproximation} be the function approximation with $\vec{\alpha}$ the solution of \eqref{eq:split} and $\tilde{f}_{L}$ the perturbed approximation with $\tilde{\vec{\alpha}}$ the solution of \eqref{eq:perturbed_split}. Then
    \begin{equation}\label{eq:approximant_difference_with_perturbation}
        \|f_{L}-\tilde{f}_{L}\|_{L^{2}(\Omega)} \le c_{f} \sqrt{L} e^{-\frac{\theta}{4}T} \|f\|_{L_{\infty}(\Omega)},
    \end{equation}
    with $c_{f}^{2} = \frac{N(1)(1+2 \nu c_{q})^{d} \pi^{d/2} q_{1}^{d}(1+4M_{d}^{2}\nu^{2}c_{q}^{2})^{2\tau}}{\nu^{d}(C_{d}c_{\Phi})^{2} \Gamma(\frac{d}{2}+1)}$.
\end{theorem}
\begin{proof}
    First, we want to explicitly state the sum over the levels, highlighting the difference from the coefficients of the two solutions:
    \begin{align*}
        \|f_{L}-\tilde{f}_{L}\|_{L^{2}(\Omega)}^{2} & = \left\|\sum_{\vec{x}_{i}^{\ell} \in \cup_{\ell=1}^{L} X_{\ell}} \alpha_{i}^{\ell}\Phi_{\delta_{\ell}}(\vec{x}_{i}^{\ell}-\cdot)-\sum_{\vec{x}_{i}^{\ell} \in \cup_{\ell=1}^{L} X_{\ell}}\tilde{\alpha}_{i}^{\ell}\Phi_{\delta_{\ell}}(\vec{x}_{i}^{\ell}-\cdot)\right\|_{L^{2}(\Omega)}^{2} \\
        & = \left\|\sum_{\vec{x}_{i}^{\ell} \in \cup_{\ell=1}^{L} X_{\ell}} (\alpha_{i}^{\ell}-\tilde{\alpha}_{i}^{\ell})\Phi_{\delta_{\ell}}(\vec{x}_{i}^{\ell}-\cdot)\right\|_{L^{2}(\Omega)}^{2}  = \left\| \sum_{\ell=1}^{L} \sum_{i=1}^{N(\ell)} (\alpha_{i}^{\ell}-\tilde{\alpha}_{i}^{\ell})\Phi_{\delta_{\ell}}(\vec{x}_{i}^{\ell}-\cdot)\right\|_{L^{2}(\Omega)}^{2} 
    \end{align*}
    Then, we will directly work with the norm to have
    \begin{align*}
       & \int_{\Omega} \left|\sum_{\ell=1}^{L} \sum_{i=1}^{N(\ell)} (\alpha_{i}^{\ell}-\tilde{\alpha}_{i}^{\ell})\Phi_{\delta_{\ell}}(\vec{x}_{i}^{\ell}-\vec{x}) \right|^{2} d\vec{x}  \le \int_{\Omega} L \sum_{\ell=1}^{L} \left|\sum_{i=1}^{N(\ell)} (\alpha_{i}^{\ell}-\tilde{\alpha}_{i}^{\ell})\Phi_{\delta_{\ell}}(\vec{x}_{i}^{\ell}-\vec{x}) \right|^{2} d\vec{x} \\
        & \le L \sum_{\ell=1}^{L} \int_{\Omega} \left|\sum_{i=1}^{N(\ell)} (\alpha_{i}^{\ell}-\tilde{\alpha}_{i}^{\ell})\Phi_{\delta_{\ell}}(\vec{x}_{i}^{\ell}-\vec{x}) \right|^{2} d\vec{x} 
         = L \sum_{\ell=1}^{L} \left\| \sum_{i=1}^{N(\ell)} (\alpha_{i}^{\ell}-\tilde{\alpha}_{i}^{\ell})\Phi_{\delta_{\ell}}(\vec{x}_{i}^{\ell}-\cdot)\right\|_{L^{2}(\Omega)}^{2}.
    \end{align*}
    Thus we have
    \begin{equation*}
        \|f_{L}-\tilde{f}_{L}\|_{L^{2}(\Omega)}^{2} \le L \sum_{\ell=1}^{L} \int_{\Omega}  \sum_{i=1}^{N(\ell)}  \sum_{j=1}^{N(\ell)} (\alpha_{j}^{\ell}-\tilde{\alpha}_{j}^{\ell})(\alpha_{i}^{\ell}-\tilde{\alpha}_{i}^{\ell})\Phi_{\delta_{\ell}}(\vec{x}_{j}^{\ell}-\vec{x})\Phi_{\delta_{\ell}}(\vec{x}_{i}^{\ell}-\vec{x}) d\vec{x}
    \end{equation*}
    Introducing $p_{i}(\vec{x}) = (\alpha_{i}^{\ell}-\tilde{\alpha}_{i}^{\ell})\Phi_{\delta_{\ell}}(\vec{x}_{i}^{\ell}-\vec{x})$, for we have that $p_{i}(\vec{x})p_{j}(\vec{x}) \neq 0$ only if $\|\vec{x}_{i}^{\ell}-\vec{x}_{j}^{\ell}\|_{2} \le 2\delta_{\ell}$. Indeed, as for instance in \eqref{eq:rowcost} we observe that for fixed $i$ the number of choices for $j$, such that $p_{i}(\vec{x})p_{j}(\vec{x}) \neq 0$, is bounded by $(1+2\nu c_{q})^{d}$. Thus, given the set of non-zero pairs
    \begin{equation*}
        \{(i,j) \, : \, \|\vec{x}_{i}^{\ell}-\vec{x}_{j}^{\ell}\|_{2} \le 2\delta_{\ell}, \, 1\le i \le N(\ell), \, 1\le i \le N(\ell)\},
    \end{equation*}
    the occurrence of a fixed index $\iota$ is bounded by $2(1+2\nu c_{q})^{d}$. Therefore,
    \begin{equation*}
        \sum_{i=1}^{N(\ell)} \sum_{j=1}^{N(\ell)} p_{i}(\vec{x})p_{j}(\vec{x}) \le \frac{1}{2} \sum_{i=1}^{N(\ell)} \sum_{\vec{x}_{j} \in B_{2\delta_{\ell}}(\vec{x}_{i})} p_{i}^{2}(\vec{x}) + p_{j}^{2}(\vec{x}) \le (1+2 \nu c_{q})^{d} \sum_{\iota=1}^{N(\ell)} p_{\iota}^{2}(\vec{x}).
    \end{equation*}
    Consequently,
    \begin{align*}
        \|f_{L}-\tilde{f}_{L}\|_{L^{2}(\Omega)}^{2} & \le L \sum_{\ell=1}^{L} \int_{\Omega} (1+2 \nu c_{q})^{d} \sum_{i=1}^{N(\ell)} (\alpha_{i}^{\ell}-\tilde{\alpha}_{i}^{\ell})^{2}\Phi_{\delta_{\ell}}^{2}(\vec{x}_{i}^{\ell}-\vec{x}) \\
        & \le (1+2 \nu c_{q})^{d} L \sum_{\ell=1}^{L} \sum_{i=1}^{N(\ell)} \int_{B_{\delta_{\ell}}(\vec{x}_{i})} (\alpha_{i}^{\ell}-\tilde{\alpha}_{i}^{\ell})^{2}\Phi_{\delta_{\ell}}^{2}(\vec{x}_{i}^{\ell}-\vec{x}) \\
        & \le (1+2 \nu c_{q})^{d} L \sum_{\ell=1}^{L} \sum_{i=1}^{N(\ell)} (\alpha_{i}^{\ell}-\tilde{\alpha}_{i}^{\ell})^{2} \delta_{\ell}^{-2d} \text{Vol}(B_{\delta_{\ell}}(\vec{x}_{i})) \\
        & = \frac{(1+2 \nu c_{q})^{d}\pi^{d/2}}{\Gamma(\frac{d}{2}+1)} L \sum_{\ell=1}^{L} \delta_{\ell}^{-d} \sum_{i=1}^{N(\ell)} (\alpha_{i}^{\ell}-\tilde{\alpha}_{i}^{\ell})^{2}.
    \end{align*}
    At this point, we would like to exploit the result of Lemma \ref{lem:alphabound} level-wise. However, Corollary \ref{cor:deltaalphabound} provides
    \begin{align*}
        \|f_{L}-\tilde{f}_{L}\|_{L^{2}(\Omega)} &\le \frac{(1+2 \nu c_{q})^{d/2}\pi^{d/4}}{\Gamma^{1/2}(\frac{d}{2}+1)} \sqrt{L} \left(\sum_{i=1}^{N} (\alpha_{i}^{\ell}(\delta)-\tilde{\alpha}_{i}^{\ell}(\delta))^{2}\right)^{1/2} \\ 
        & \le \frac{\sqrt{N(1)}(1+2 \nu c_{q})^{d/2} \pi^{d/4} q_{1}^{d/2} (1+4M_{d}^{2}\nu^{2}c_{q}^{2})^{\tau}}{\nu^{d/2}C_{d}c_{\Phi}\Gamma^{1/2}(\frac{d}{2}+1)} \sqrt{L} e^{-\frac{\theta}{4}T}\|f\|_{L_{\infty}(\Omega)},
    \end{align*}
    which yields the desired result.
\end{proof}

\begin{remark}
    In practice, we choose $T = \frac{-4dL^{2}}{\theta} \ln(c_{h}\mu)$. Indeed, thid is the smallest value for which the convergence holds.
    Moreover, it is worth to notice that, with this choice, we have
    \begin{equation*}
        \sqrt{L} e^{-\frac{\theta}{4}T} = \sqrt{L} (c_{h}\mu)^{dL^{2}},
    \end{equation*}
    which is clearly dominated by the exponential.
    Additionally, the assumption of Theorem \ref{thm:decayerror} lead to
    \begin{equation*}
        T^{d} \le e^{\frac{\theta}{2} T} \iff \frac{-4dL^{2}}{\theta} \ln(c_{h}\mu) \le (c_{h} \mu)^{-2L^{2}}.
    \end{equation*}
    which holds for $L^{2} \ge -\frac{\ln(\frac{2d}{\theta})}{\ln(c_{h} \mu)}$. Indeed,
    \begin{equation*}
        \frac{2d}{\theta} \ln((c_{h}\mu)^{-2L^{2}}) \le \frac{2d}{\theta} (c_{h} \mu)^{-L^{2}} \le (c_{h} \mu)^{-2L^{2}}.
    \end{equation*}
    With the setup of the numerical experiments in Section \ref{sec:numerics} we have that this is equivalent to ask for $L \ge 3$.
    Moreover, we have that 
    \begin{equation*}
        (c_{h} \mu)^{-dL} \ge (CC_{\Sigma})^{2} 2 L e^{-\theta} c_{q}^{2d/L-d}.
    \end{equation*}
\end{remark}

Lastly we are in the position to state the comprehensive result on the convergence of the multiscale approximation scheme with the additional matrix truncation in the Jacobi method. 

\begin{corollary}\label{cor:approximant_difference_with_perturbation}
    Under the hypothesis of Theorem \ref{thm:approximant_difference_with_perturbation} and Theorem \ref{thm:errorbound} we have that 
    \begin{equation}\label{eq:truncated_errorbound}
        \|f-\tilde{f}_{L}\|_{L^{2}(\Omega)} \le (C_{2}(C_{1}\mu^{\tau})^{L} + c_{f} C_{4} \sqrt{L} e^{-\frac{\theta}{4}T}) \|f\|_{H^{\tau}(\Omega)}.
    \end{equation}
    Additionally, the result can be rewritten as 
    \begin{equation}\label{eq:truncated_errorbound_N}
        \|f-\tilde{f}_{L}\|_{L^{2}(\Omega)} \le \left(C_{2}C_{1}^{L} c_{q}^{\tau}\left(\frac{N(1)}{N}\right)^{\frac{\tau}{d}}+C_{f} \sqrt{L} c_{h}^{dL^{2}} c_{q}^{dL} \left(\frac{N(1)}{N}\right)^{L}\right) \|f\|_{H^{\tau}(\Omega)}.
    \end{equation}
\end{corollary}
\begin{proof}
    Theorem \ref{thm:errorbound} and Theorem \ref{thm:approximant_difference_with_perturbation}  imply that
    \begin{align*}
        \|f-\tilde{f}_{L}\|_{L^{2}(\Omega)} &\le \|f-f_{L}\|_{L^{2}(\Omega)} + \|f_{L}-\tilde{f}_{L}\|_{L^{2}(\Omega)} \\
        &\le C_{2}(C_{1}\mu^{\tau})^{L} \|f\|_{H^{\tau}(\Omega)}+c_{f} \sqrt{L} e^{-\frac{\theta}{4}T}\|f\|_{L_{\infty}(\Omega)} \\
        &\le C_{2}(C_{1}\mu^{\tau})^{L} \|f\|_{H^{\tau}(\Omega)}+c_{f} C_{4} \sqrt{L} e^{-\frac{\theta}{4}T}\|f\|_{H^{\tau}(\Omega)},
    \end{align*}
    which proves the stated result. The estimate
    \eqref{eq:truncated_errorbound_N} follows from Corollary \ref{cor:errorbound} and Assumption \ref{ass:pointset} with $T = \frac{-4dL^{2}}{\theta} \ln(c_{h}\mu)$, which implies
    \begin{equation*}
             e^{-\frac{\theta}{4} T} = (c_{h} \mu)^{dL^{2}} \le c_{h}^{dL^{2}} c_{q}^{dL} \left(\frac{N(1)}{N}\right)^{L}.
    \end{equation*}
\end{proof}

Corollary \ref{cor:approximant_difference_with_perturbation} extends the result of Theorem \ref{thm:errorbound} to highlight that, in the long run, the approximation quality does not suffer from the truncation procedure on the Jacobi solution. Next, we will quantify the computational gain of such procedure.

\subsection{Computational cost}\label{subsec:computationalcost}

Having convergence of the perturbed system, we are in the position to look closer to its cost.

\begin{lemma}\label{lem:nnzMt}
    The number of non-zero entries in $\tilde{M}(T)$ given in \eqref{eq:perturbedmatrix} is bounded by
    \begin{equation*}\label{eq:nnzMt}
        \frac{N(0)c_{q}^{d}}{c_{\#}}(1+T)^{d} N 
    \end{equation*}
\end{lemma}
\begin{proof}
    The number of non-zero entries per row in  $\tilde{\mathfrak{X}}_{\ell j}(T)$ is bounded by
    \begin{align*}
    	\# \{ \vec{x}^{(j)}_{k} \in B_{Tq_j} (\vec{x}^{(\ell)}_{i}) \} \le \left(\frac{q_{j}+Tq_{j}}{q_{j}}\right)^{d}\le (1+T)^{d}.
    \end{align*}
    
    Therefore, exploiting \eqref{eq:Nmu} and \eqref{eq:Nmu2}, the number of non-zero entries for the $\tilde{M}$ matrix is bounded by
    
    \begin{align*}
        \sum_{j=1}^{L-1} \sum_{\ell = j+1}^{L} N(\ell) (1+T)^{d} & \le N(0)c_{q}^{d}(1+T)^{d} \sum_{j=1}^{L-1} \sum_{\ell = j+1}^{L} \mu^{-d\ell} \\
        & = \frac{N(0)c_{q}^{d}(1+T)^{d}}{1-\mu^{d}} \left[\mu^{-dL} - \sum_{j=1}^{L-1} \mu^{-dj} \right].
    \end{align*}
    Moreover, we have
    \begin{align*}
        \mu^{-dL} - \sum_{j=1}^{L-1} \mu^{-dj} & = \sum_{j=1}^{L} \mu^{-dj} - 2\sum_{j=1}^{L-1} \mu^{-dj} = \sum_{j=1}^{L} \mu^{-dj} - 2\left(\frac{\sum_{j=1}^{L} \mu^{-dj}}{\mu^{-d}} -1 \right)  \\
        & = \left( 1-2\mu^{d} \right)\sum_{j=1}^{L} \mu^{-dj} +2  \le \left( 1-2\mu^{d} \right)\sum_{j=1}^{L} \frac{N(\ell)}{c_{\#}} +2.
    \end{align*}
    Additionally,  
    \begin{equation*}
        \frac{1-2\mu^{d}+2c_{\#}/N}{1-\mu} < 1 \quad \Leftrightarrow \quad \frac{2}{\mu^{d}} < N,
    \end{equation*} 
    holds if $c_{\#} > \frac{2N(1)}{N}$ (which is not a real limitation). Therefore, we conclude that the number of non-zero entries of $\tilde{M}$ is bounded by
    \begin{equation}
        \le \frac{N(0)c_{q}^{d}}{c_{\#}}(1+T)^{d} N.
    \end{equation}
\end{proof}

In order to present a clear bound for the computational cost of the perturbed system we introduce additional assumptions.
\begin{assumption}\label{ass:perturbedcost}
    Let $X_{N(\ell)} \subset \Omega$ for $\ell \in \N$ be a family of scattered node sets which satisfy Assumption \ref{ass:pointset}. We further assume that $c_{h} > \mu$.
\end{assumption}

\begin{theorem}\label{thm:perturbed_cost}
    Let $\varepsilon \in (0,1)$ be given. The solution of the perturbed system \eqref{eq:perturbed_split} is obtained at the cost
	\begin{equation*}
		\operatorname{costs} (\text{solve}_{\eqref{eq:perturbed_split}}) \in \mathcal{O} \left( L^{d}N\log^{d}(N) +N\lceil \frac{1}{2} \sqrt{C_{cg}} \log{\left(\frac{ (C_{d}c_{\Phi})^{-1} \left(1 +4 M_{d}^{2} \nu^{2} c_{q}^{2}\right)^{\tau}) \sqrt{L}}{\varepsilon} \right)} \rceil \right).
	\end{equation*}
\end{theorem}
\begin{proof}
    As pointed out before, the computational cost of the Jacobi iteration is dominated by the cost of the matrix vector multiplication. However, the matrix $\tilde{M}$ performs strictly better than $M$. Indeed, taking into account Lemma \ref{lem:nnzMt} with $T = \frac{-4dL^{2}}{\theta} \ln(c_{h}\mu)$ we can already note the expected behaviour:
    \begin{equation*}
        \frac{N(0)c_{q}^{d}}{c_{\#}}(1+T)^{d}N = \frac{N(0)c_{q}^{d}}{c_{\#}}\left(1+(4L/\theta)\log\left((c_{h} \mu)^{-dL}\right)\right)^{d} N \sim cL^{d}\log^{d}(N)N.
    \end{equation*}
    Precisely, let Assumption \ref{ass:perturbedcost} be satisfied, moreover, we have that
    \begin{equation*}
        N > N(L) \ge c_{\#}\mu^{-dL}.
    \end{equation*}
    Hence, we obtain
    \begin{align*}
        \frac{N(0)c_{q}^{d}}{c_{\#}}(1+T)^{d}N &= \frac{N(0)c_{q}^{d}}{c_{\#}}\left(1+(4L/\theta)\log\left(c_{h}^{-dL} \mu^{-dL}\right)\right)^{d} N \\
        &\le \frac{N(0)c_{q}^{d}}{c_{\#}}\left(1+(8L/\theta)\log\left(\mu^{-dL}\right)\right)^{d} N \\
        &\le \frac{N(0)c_{q}^{d}}{c_{\#}}\left(1+(8L/\theta)\left[\log(N) -\log(c_{\#})\right]\right)^{d} N \\
        &\le \frac{N(0)c_{q}^{d}}{c_{\#}}((9L/\theta)^{d}\log^{d}(N) N.
    \end{align*}
    Thus
    \begin{equation}\label{eq:jacobicost}
        \text{cost}(\text{Jacobi}) = \mathcal{O}(L^{d}\log^{d}(N)N),
    \end{equation}
    which combined with Theorem \ref{thm:costsforeps} yields the desired result.
\end{proof}

\section{Numerical Experiments}\label{sec:numerics}
In this section we look at numerical evidence of our theoretical results. In our experiments we chose as sequence of sets $X_{\ell}$, the regular grids on $\Omega = [0,1]^{2}$ with size $2^{-\ell}$. Details of the sets are given in Table \ref{tab:datasets}. 
\begin{table}[ht]
    \centering
    \begin{tabular}{c|c|c|c|c|c|c|c}
         & $X_{1}$ & $X_{2}$ & $X_{3}$ & $X_{4}$ & $\dots$ & $X_{10}$ & $X_{11}$\\
         \hline
        $N(\ell)$ & 9 & 25 & 81 & 289 & $\dots$ & 1050625 & 4198401 \\
        $h_{\ell}$ & 0.354 & 0.177 & 0.0884 & 0.0442 & $\dots$ & 0.000690 & 0.000345 \\
        $q_{\ell}$ & 0.25 & 0.125 & 0.0625 & 0.0313 & $\dots$ & 0.000488 & 0.000244
    \end{tabular}
    \caption{Details of sets $X_{\ell}$}
    \label{tab:datasets}
\end{table}
Then, we chose $\mu = 0.5$ and $\nu = 4$ according to Assumption \ref{ass:pointset}. 

Our choice for compactly supported radial basis function is the Wendland function $\Phi = \phi_{(3,1)}(\|\cdot\|_{2}) \in C^{2}(\R^{3})$ generating $H^{3}(\R^{d})$, where $\phi_{(3,1)}(r) = (1-r)_{+}^{4}(4r+1)$, and our target function is the standard Franke function

\begin{equation}\label{eq:Franke}
    F(x,y) = \frac{3}{4}e^{-\frac{(9x-2)^{2}+(9y-2)^{2}}{4}}+\frac{3}{4}e^{-\frac{(9x+1)^{2}}{49}-\frac{(9y+1)}{10}} +\frac{1}{2}e^{-\frac{(9x-7)^{2}+(9y-3)^{2}}{4}}-\frac{1}{5}e^{-(9x-4)^{2}-(9y-7)^{2}}.
\end{equation}

We first investigate the quality of our bounds on the triangular matrix involved by the Jacobi method, then we analyse the implication of its truncation both from the accuracy and complexity point of view. Lastly, we show the results of a parallel implementation of the methods highlighting their scalability \cite{lotfparallel}.

The parallel simulations were executed on an NVIDIA A100 GPU which are part of the cluster MaRC3a\footnote{For a description of the cluster, see \url{https://www.uni-marburg.de/de/hrz/dienste/hochleistungsrechnen}} with CUDA 11.1. The truncation analysis was performed with Python. Indeed, on the latter the emphasis is on trend analysis rather than performance, therefore, we opted to execute the experiments exclusively on the lower levels prioritizing the development of compact and intuitive code.

\subsection{Bounds on M}
First, we examine the accuracy of the norm $\|M_{L}\|_{2 \to 2}$. We compared the theoretical bound of Lemma \ref{lem:Mbound} with the highest singular value found with the SVD decomposition of $M_{L}$. We estimated all parameters enclosed in \eqref{eq:Mbound} numerically.
Figure \ref{fig:M_norm_bound} shows that the theoretical bound for $\|M_{L}\|_{2 \to 2}$ presented in Lemma \ref{lem:Mbound} is sharp up to a constant factor. The function $f(L) = 0.48 e^{0.78 L}$ was also plotted to emphasize the slope and remark the similarity of the bound estimate with the numerical results.

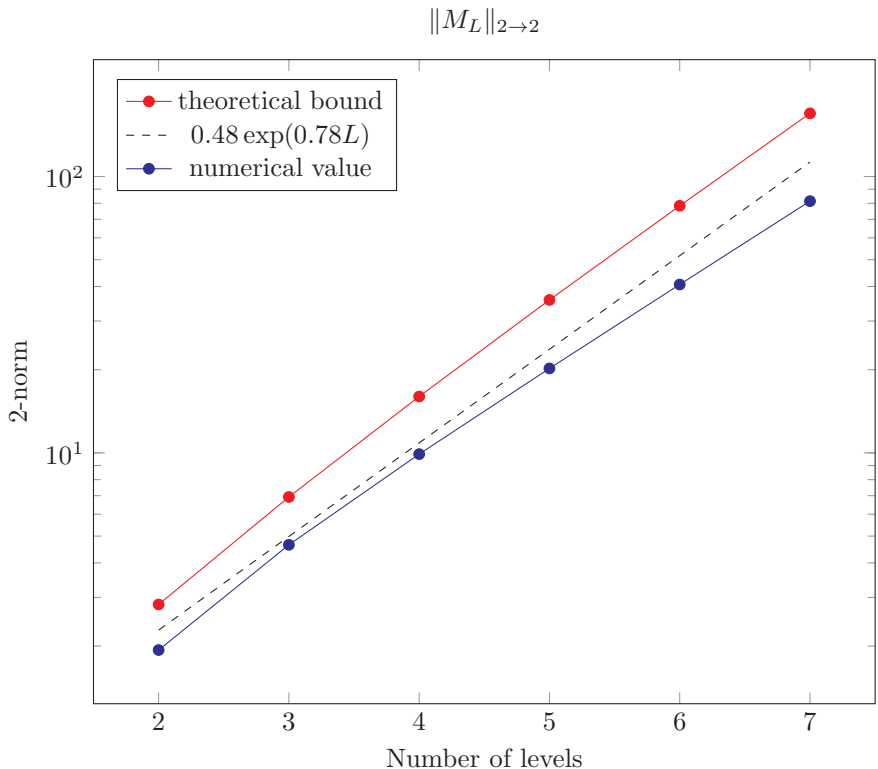
\begin{figure}[ht]
    \centering
    \begin{tikzpicture}
        \begin{semilogyaxis}[
            scale=1.5,
            legend pos=north west,
            xlabel=Number of levels, ylabel=2-norm,  title=$\|M_{L}\|_{2 \to 2}$]
        \addplot[color=red,mark=*] coordinates {
		(2, 2.828)
		(3, 6.928)
		(4, 16.0)
		(5, 35.777)
  		(6, 78.384)
		(7, 169.328)
        };
        \addlegendentry{theoretical bound}

        \addplot[color=black, dashed] coordinates {
        (2, 2.284)
		(3, 4.983)
		(4, 10.870)
		(5, 23.713)
		(6, 51.730)
		(7, 112.847)
        };
        \addlegendentry{$0.48\exp(0.78 L)$}
        
        \addplot[color=blue,mark=*] coordinates {
		(2, 1.935)
		(3, 4.649)
		(4, 9.899)
		(5, 20.212)
  		(6, 40.674)
		(7, 81.477)
        };
        \addlegendentry{numerical value}
        \end{semilogyaxis}
    \end{tikzpicture}
    \caption{Comparison of the theoretical bound of $\|M_{L}\|_{2 \to 2}$ from Lemma \ref{lem:Mbound} with its numerical value.
\label{fig:M_norm_bound}}
\end{figure}

\subsection{Truncation}
We want to further investigate the truncation influence of the system solution.
Therefore, on one side we compare the number of non-zero elements on the original matrix and on its truncated variation, since within the Jacobi method the matrix vector multiplication cost, which dominate its cost complexity, is directly tied to these values. On the other side, we compute the norm of the difference matrix $M-\tilde{M}$ to inspect how large their difference is, which is clearly related to the difference of the solutions as shown on the theoretical analysis. 

Indeed, we can see that, in Figure \ref{fig:M_truncation_norm_diff}, independently from the number of levels $L$, the normalized $2$-norm of the difference from the truncation and the original matrix drop to zero as $T$ grow. The exponential reference, $0.77 e^{-0.69 T}$ gives us an idea of the decrease rate of the normalized difference.
Moreover, we point out that, assuming $\|M_{L}\|_{2 \to 2}$ behaves as its upper bound, the ratio behaves as $T^{d}e^{-\theta T}$. 
Indeed, we can notice how different levels present the same exponential decay with respect to $T$, just with different fluctuations.

Also, Figure \ref{fig:M_truncation_nnz_ratio} suggests how the computational gain, in the form of a faster matrix vector multiplication, wanes with the growth of $T$. Indeed, eventually the ratio will reach the identity. 
As discussed before, since $M_{L}$ has a row more than $M_{L-1}$, for constant $T$, higher levels are associated with a smaller ratio. Moreover, assuming that there exists $T^{*}_{L}$ such that the non-zero elements of $\tilde{M}_{L}$ coincide with the ones from $M_{L}$ then, Assumption \ref{ass:pointset} lead to 
\begin{equation*}
    c_{q} q_{L+1} \le \mu q_{L} \le \frac{c_{q}}{c_{h}} q_{L+1}.
\end{equation*}
This means that we will roughly have $T^{*}_{L+1} \sim \frac{c_{q}}{\mu} T^{*}_{L}$. This result can be extended to the non-zeros ratio, indeed, the number of points $N(L)$ behave as $\mu^{-dL}$, from which we expect that, 
\begin{equation*}
    \frac{\text{ratio}_{L+1}}{\text{ratio}_{L}} \sim 1 \iff \left(\frac{T_{L+1}}{T_{L}}\right)^{d} \sim \frac{N(L+1)}{N(L)} \sim \mu^{-d}.
\end{equation*}
This means that to keep a given ratio increasing the number of levels, we must increase $T$ by a factor $\mu^{-1}$, which is 2 in our simulations. Since we exploited an upper bound for these considerations, we can see how $r_{L}(T)$ is always bounded by $r_{L+1}(2T)$, which is exactly what Figure \ref{fig:M_truncation_nnz_ratio} shows.

To conclude, putting all the information together, we can achieve a significant improvement of the matrix vector multiplication with a negligible accuracy loss with a careful choice of $T$. Indeed, while the normalized norm difference exponentially tends to zero, the ratio of the non-zero elements increase just linearly. 

\begin{figure}
    \centering
    \begin{tikzpicture}
        \begin{semilogyaxis}[
            scale=1.5,
            legend pos=north east,
            xlabel=T, ylabel=norm ratio,  title=$\frac{\|M_{L}-\tilde{M}_{L}\|_{2 \to 2}}{\|M_{L}\|_{2 \to 2}}$
            ]

        \addplot[color=brown,mark=*] coordinates {
        (1, 0.71423)
		(2, 0.13579)
		(3, 0.12313)
		(4, 0.02561)
		(5, 0.02839)
		(6, 0.01009)
        };
        \addlegendentry{3 levels}

        \addplot[color=blue,mark=*] coordinates {
        (1, 0.62183)
		(2, 0.14416)
		(3, 0.11283)
		(4, 0.03094)
		(5, 0.02900)
		(6, 0.01264)
        };
        \addlegendentry{4 levels}

        \addplot[color=green,mark=*] coordinates {
		(1, 0.57363)
		(2, 0.14844)
		(3, 0.10547)
		(4, 0.03408)
		(5, 0.02732)
		(6, 0.01421)
        };
        \addlegendentry{5 levels}

        \addplot[color=orange,mark=*] coordinates {
        (1, 0.54880)
		(2, 0.15050)
		(3, 0.10077)
		(4, 0.03577)
		(5, 0.02591)
		(6, 0.01507)
        };
        \addlegendentry{6 levels}

        \addplot[color=red,mark=*] coordinates {
        (1, 0.53714)
		(2, 0.15138)
		(3, 0.09821)
		(4, 0.03653)
		(5, 0.02513)
		(6, 0.01546)
        };
        \addlegendentry{7 levels}

        \addplot[color=black, dashed] coordinates {
		(1, 0.38674)
		(2, 0.19398)
		(3, 0.09730)
		(4, 0.04880)
		(5, 0.02448)
		(6, 0.01228)
        };
        \addlegendentry{$0.77 \cdot e^{-0.69 T}$}
        \end{semilogyaxis}
    \end{tikzpicture}
    \caption{normalized norm difference from $M_{L}$ and $\tilde{M}_{L}$. \label{fig:M_truncation_norm_diff}}
\end{figure}
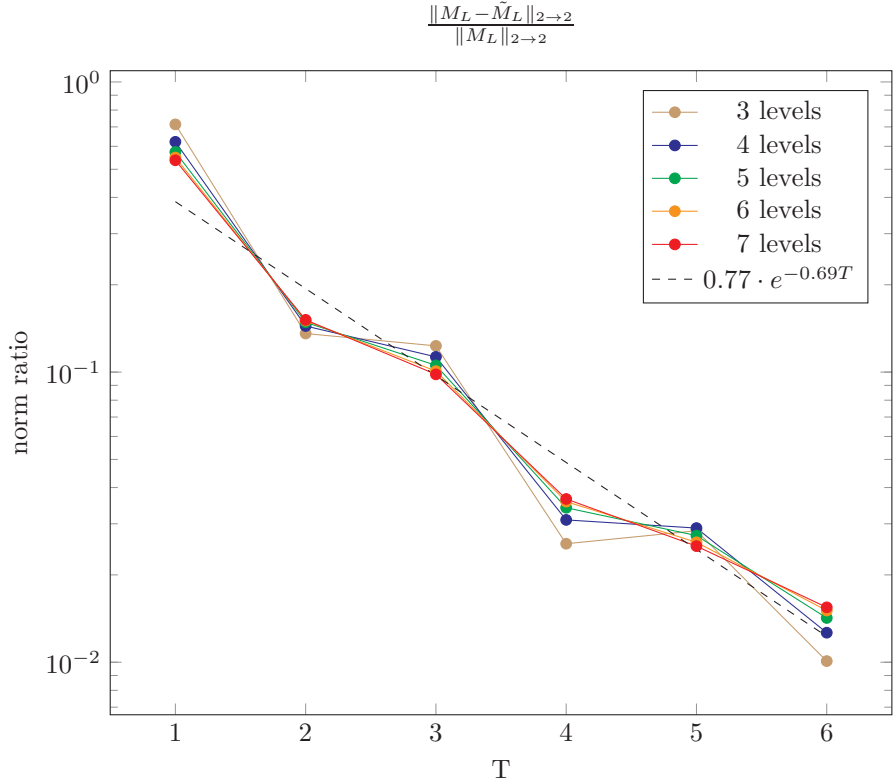

\begin{figure}
    \centering
    \begin{tikzpicture}
        \begin{axis}[
            scale=1.5,
            legend pos=north west,
            xlabel=T, ylabel=nnz ratio,  title=$\frac{nnz(\tilde{M}_{L})}{nnz(M_{L})}$
            ]

        \addplot[color=brown,mark=*] coordinates {
		(1, 0.03714)
		(2, 0.19105)
		(3, 0.33960)
		(4, 0.54541)
		(5, 0.65638)
		(6, 0.79597)
        };
        \addlegendentry{3 levels}

        \addplot[color=blue,mark=*] coordinates {
        (1, 0.01914)
		(2, 0.09400)
		(3, 0.17522)
		(4, 0.29077)
		(5, 0.36789)
		(6, 0.47539)
        };
        \addlegendentry{4 levels}

        \addplot[color=green,mark=*] coordinates {
        (1, 0.00839)
		(2, 0.03983)
		(3, 0.07685)
		(4, 0.13027)
		(5, 0.17197)
		(6, 0.22975)
        };
        \addlegendentry{5 levels}

        \addplot[color=orange,mark=*] coordinates {
        (1, 0.00440)
		(2, 0.02038)
		(3, 0.04036)
		(4, 0.06932)
		(5, 0.09454)
		(6, 0.12878)
        };
        \addlegendentry{6 levels}

        \addplot[color=red,mark=*] coordinates {
        (1, 0.00291)
		(2, 0.01321)
		(3, 0.02669)
		(4, 0.04622)
		(5, 0.06457)
		(6, 0.08905)
        };        
        \addlegendentry{7 levels}

        \end{axis}
    \end{tikzpicture}
    \caption{non-zero elements ratio from $M_{L}$ and $\tilde{M}_{L}$.\label{fig:M_truncation_nnz_ratio}}
\end{figure}
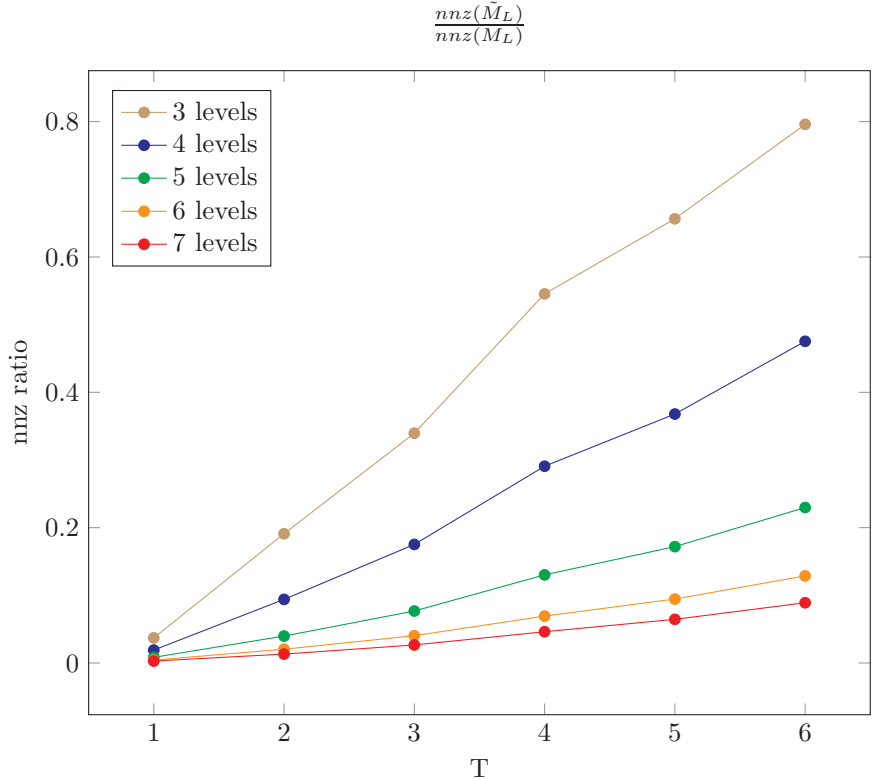

\subsection{Parallelization on a GPU}
The transition from serial to parallel programming in the last decades was inevitable \cite{KungleLance05}. Indeed, addressing real world applications in a serial environment is no longer feasible, while application of GPU accelerated methods is becoming increasingly popular, see e.g. \cite{Liu15}\cite{de2020gpu}.

In its classical formulation, the multi-scale method, requires $L$ sequential solutions via the CG methods of linear systems of exponentially growing size. Indeed, the first steps may limit the effectiveness of high computational resources available, as every step is strictly dependent on the previous ones. On the other hand, our reformulation allows the system solution with just two steps: the Jacobi \eqref{eq:split_system_triangular} and the CG solution of \eqref{eq:blockdiagonal}. Both can highly benefit from a parallel solution, since for both the higher computational cost comes from the matrix-vector multiplication which involves large matrices. We start with an efficient implementation of the latter.

In the parallel solution of \eqref{eq:blockdiagonal_levelwise}, we solve $L$ independent CG systems geometrically increasing in size. Having said that, the number of iterations needed for each cg solution is bounded independently by the level. Therefore, our aim is to distribute the available resources proportionally to the matrix size, solving multiple instances simultaneously.

\begin{algorithm}[ht]
    \caption{Parallel Conjugate Gradient\label{alg:cg}}
    \begin{algorithmic}[1]
    
    \State \textbf{Input:} Points array $X$, rhs array $\bm{f}$, cumulative number of points array $M$, Wendland coefficients, threshold $\varepsilon$, number of levels $L$.
    
    \State \textbf{Output:} Solution array of the CG.

    \For{each level $\ell$ on different streams}
        \State Set $X_{\ell}$, $d \cdot N(\ell)$-array with points of level $\ell$;
        \State Set $N(\ell)$-arrays: $\text{res}_{\ell}$, $\text{dir}_{\ell}$, $\text{temp}_{\ell}$, $\text{sol}_{\ell}$;
        \State $\text{res}_{\ell} \gets \text{rhs}_{\ell}$; $\text{dir}_{\ell} \gets \text{rhs}_{\ell}$;
        \State $\text{r\_norm}[\ell] \gets \text{cublasDdot}(\text{res}_{\ell}, \text{res}_{\ell})$;
    \EndFor

    \While{$\text{r\_norm}[\ell] < \varepsilon$ for all $\ell$}
        \For{each level $\ell$ (on stream $\ell$)}
            \State Set $\text{temp}_{\ell} \gets 0$;
            \State $\text{temp}_{\ell} \gets \text{AMV}(X_{\ell}, N(\ell), \log_2(N(\ell))+1, M, \text{dir}_{\ell}, \text{temp}_{\ell}, \nu h_{0} \mu^{\ell})$;
            \State $\alpha[\ell] \gets \text{cublasDdot}(\text{dir}_{\ell}, \text{temp}_{\ell})$;
            \State Synchronize;
            \State $\alpha[\ell] \gets \text{r\_norm}[\ell] \cdot \alpha[\ell]$;
            \State $\text{sol}_{\ell} \gets \text{cublasDaxpy}(\alpha[\ell], \text{dir}_{\ell}, \text{sol}_{\ell})$;
            \State $\text{res}_{\ell} \gets \text{cublasDaxpy}(-\alpha[\ell], \text{temp}_{\ell}, \text{res}_{\ell})$;
            \State $\text{r2\_norm}[\ell] \gets \text{cublasDdot}(\text{res}_{\ell}, \text{res}_{\ell})$;
            \State Synchronize;
            \State $\beta[\ell] \gets \text{r2\_norm}[\ell] / \text{r\_norm}[\ell]$;
            \State $\text{cublasDscalar}(\text{dir}_{\ell}, \beta[\ell])$;
            \State $\text{dir}_{\ell} \gets \text{cublasDaxpy}(1, \text{res}_{\ell}, \text{dir}_{\ell})$;
            \State $\text{r\_norm}[\ell] \gets \text{r2\_norm}[\ell]$;
        \EndFor
    \EndWhile

    \end{algorithmic}
\end{algorithm}

Algorithm \ref{alg:cg} outlines the parallel solution of the system. First, the dependencies are organized such that the available computational resources are effectively exploited. Then, a stream (all processes within a stream act independently from processes for other streams) for each level is created. Lastly, the cg routines are executed in parallel over the streams, where synchronization blocks are placed to guarantee the correct computation of coefficients and improve the concurrency between streams. 
The matrix-vector multiplication (AMV) is a critical step in the cg routine. Indeed, the performance of the routine largely depends on its effective implementation. Consequently, taking into account the structure of the matrices $A_{\ell}$, leads to significant improvements. In our settings, there are two possible strategies: the kd-tree organization of the point sets or the sparse storage of the matrices entries.
An extensive discussion on the matter is out of scope of the manuscript, the former is implemented. 
This choice allow us to have a matrix vector multiplication with computational cost $\mathcal{O}(N\log N)$ on a serial setting and $\mathcal{O}(N/p\log N)$ on a parallel environment with $p$ processors.

\begin{algorithm}[ht]
    \caption{Parallel Jacobi method \label{alg:jb}}
    \begin{algorithmic}[1]
    
    \State{Input: Points array $X$, rhs array $\bm f$, cumulative number of points array $M$, Wendland coefficients, $\varepsilon$, number of levels $L$.}
    \State{Output: Solution array of Jacobi iterative method.}
    \vspace{0.1cm}
    \State $\bm \beta_{i} \leftarrow \bm f$; 
    \For{$i = 0, \dots, L-1$}   
        \State For $0 \le  \ell  \le L-2$ (in parallel): 
        \State $\bm t^{(\ell)} \leftarrow$ CG($A_{\ell}, \bm \beta_{i}^{(\ell)}$); 
        \State $\bm \beta_{i} \leftarrow \bm f +$ McV($B_{\ell}, \bm t^{(\ell)}$); 
    \EndFor
    \end{algorithmic}
\end{algorithm}

As shown in Algorithm \ref{alg:jb} also the parallel implementation of the Jacobi method, revolves around its matrix-vector multiplication (McV). However, differently from the $A_{\ell}$ matrices, $M$ is not a sparse matrix except the zero blocks. Assuming a lack of space to store the matrix values, we cannot compute the matrices in advance (which might be a valid, but expensive, opportunity). Therefore, to compute the matrix vector multiplication we need first to solve the systems

\begin{equation*}
    D_{L-1} \vec{\beta}^{(t)} = \vec{\beta}^{(m)}, \quad \text{i.e. }\quad A_{j} \vec{\beta}^{(t)}_{j} = \vec{\beta}^{(m)}_{j} \quad 1 \le j \le L-1
\end{equation*}

which can be done in parallel over $j$. Then, the Jacobi solution is updated

\begin{equation*}
    \vec{\beta}^{(m+1)} = B_{L} \vec{\beta}^{(t)} + \mathbf{f},
\end{equation*}

where $B_{L} = T_{L}-D_{L}$.

Indeed, the structure of the matrices $B_{jk}$ and $A_{j}$ allow an effective and efficient computation. Moreover, another solution is to consider the perturbed matrix \eqref{eq:perturbedmatrix}, and trade some accuracy on the solution for speed up of the computation.

To close the algorithm overview Table \ref{tab:complexity} present the complexity for Algorithm \ref{alg:cg} and \ref{alg:jb}.

\begin{table}[ht]
    \centering
    \begin{tabular}{c|c|c}
         & Space & Time \\
         \hline
        Conj. Grad. & $\mathcal{O}((d+4)N)$ & $\mathcal{O}(kN\log_{2}N)$ \\
        Jacobi & $\mathcal{O}((d+1)N+L)$ & $\mathcal{O}(N\log_{2}N)$ \\
    \end{tabular}
    \caption{Complexity of algorithms}
    \label{tab:complexity}
\end{table}

Is worth to mention that the provided code works as long as the data can be indexed with positive \verb|int|s, i.e. until $dN \le 2^{32}-1$ (which are roughly 16Gb). The extension to different types is beyond the scope of this article. However, the treatment of larger data sets will be addressed in a subsequent publication. 

Lastly, we present the numerical result of the interpolation time and efficiency over different set sizes. We solve system \eqref{eq:split}, with the sequence of dataset as shown in Table \ref{tab:datasets} and we choose as target function the Franke function \eqref{eq:Franke}. Both, Figure \ref{fig:system_solution} and Figure \ref{fig:system_efficiency}, have "Number of warps" on the x-axis, since the most meaningful measurements are taken where the processors are fully split on warps. However, the plot can be stretched by a factor 32 (which is the size of a single warp in our setting) to have the actual number of processors employed.
The time is taken with respect to the CPU clock, therefore we expect sub-optimal decrease of the time with respect to the number of processors, since it keeps track also of time devoted to data management. Moreover, the code was handwritten to fully exploit the strengths of our setup, leaving the optimization aside.

\begin{figure}[ht]
    \centering
    \begin{tikzpicture}
    	\begin{loglogaxis}[
            scale=1.5,
            legend pos=north east ,
            xlabel=Number of warps, ylabel=Time (s),  title=Resolution time of the system solution]
    
    	\addplot[color=blue,mark=*] coordinates {
            (3456, 0.974)
            (1728, 0.974)
            (864, 0.974)
            (432, 1.197)
            (216, 1.519)
            (108, 2.675)
            (54, 5.017)
            (27, 8.802)
        };
        \addlegendentry{$7$ levels (22k)}
    
        \addplot[color=orange,mark=*] coordinates {
            (3456, 1.485)
            (1728, 1.642)
            (864, 2.037)
            (432, 3.534)
            (216, 6.463)
            (108, 11.445)
            (54, 22.156)
            (27, 39.851)
    	};
        \addlegendentry{$8$ levels (88k)}
    
        \addplot[color=brown,mark=*] coordinates {
            (3456, 4.479)
            (1728, 5.101)
            (864, 8.339)
            (432, 14.601)
            (216, 28.157)
            (108, 52.544)
            (54, 100.949)
            (27, 188.368)
    	};
        \addlegendentry{$9$ levels (350k)}
    
        \addplot[color=violet,mark=*] coordinates {
            (3456, 20.111)
            (1728, 20.981)
            (864, 34.662)
            (432, 65.980)
            (216, 128.995)
            (108, 245.286)
            (54, 479.171)
            (27, 906.343)
    	};
        \addlegendentry{$10$ levels (1.4M)}
    
        \addplot[color=red,mark=*] coordinates {
            (3456, 92.642)
            (1728, 93.722)
            (864, 156.326)
            (432, 298.409)
            (216, 588.660)
            (108, 1152.734)
            (54, 2286.841)
            (27, 4351.735)
    	};
        \addlegendentry{$11$ levels (5.6M)}

        \addplot[color=black, dashed] coordinates {
			(1728,  35.935)
			(864,  64.870)
			(432, 122.741)
			(218, 236.358)
			(108, 469.963)
			(54, 932.926)
			(27, 1858.852)
        };
        \addlegendentry{$\frac{5\cdot 10^{4}}{\# \text{warp}}+7$}
    
        \end{loglogaxis}
    \end{tikzpicture}
    \caption{Time required for solve the system \eqref{eq:split} with different data sizes
    \label{fig:system_solution}}
\end{figure}
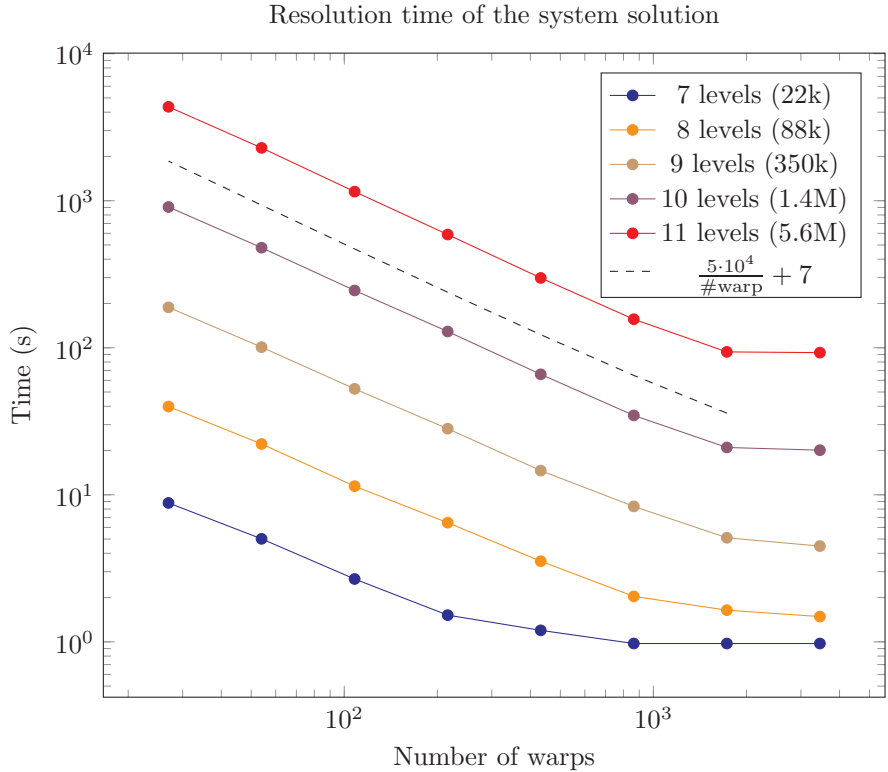

As can be seen in Figure \ref{fig:system_solution}, the computational time decrease linearly as the number of processors increase. As the management time is almost fix independently of the problem size, small problems cannot benefit from a large amount of computational power, indeed the blue line bends, with respect to its initial trajectory, as the number of warps involved increase. On the other hand, larger problems benefit from a large number of processors involved. 
The dashed function stress that all the lines on the graph are almost parallel, emphasizing their slope. Moreover, highlight how doubling the number of processors halves the time needed to solve the system. Additionally, is worth to remark how, independently from the level (or the data size), there is almost no improvement for the maximum number of warps. This is motivated by the fact that the parallel structure is exploited mainly during the matrix-vector multiplication (where the benefits of a larger number of processors can be seen). However, its impact on the system solution is already negligible, i.e. the time we measure is mostly spent on resource coordination and scheduling.

Additionally, we tested the largest problem for this settings, consisting of 14 levels and $350M$ points in $\R^2$. The time needed to solve the system with the maximum number of warps (3456) was of almost $2.5$ hours. 

In a parallel setting, the best we can hope for is to split the workload equally between the processors. If we call the serial runtime $T_{s}$ and the parallel runtime $T_{p}$, we aim to have a relation like $T_{p} = T_{s}/p$ where $p$ represent the number of processors involved. Furthermore, as $p$ increase, more parallel overhead is expected from the computations, therefore we expect that larger $p$ will not always provide advantages. To monitor the optimal amount of processors and  their effectiveness, we compute the efficiency of the parallel program
\begin{equation*}
     E = \frac{T_{s}}{p T_{p}} .   
\end{equation*}
Clearly, effective parallel programs have an efficiency closer to $1$, even though in practice this is hardly achievable, especially when the task to parallelization is not straightforward.
However, what is more meaningful is the concept of scalability \cite{Pacheco22}. We say that a program is strongly scalable if its efficiency increases with the number of processors while the problem size is kept fixed. On the other hand, if the efficiency is kept fixed as the number of processors as well as the problem size increases accordingly, then the program is said to be weakly scalable.
Scalability is a key concept in parallel programming, since a scalable program can be solve in almost arbitrary small time, given enough parallel resources.

Figure \ref{fig:system_efficiency} shows the efficiency plots for different problem sizes and different amount of processors involved. 

\begin{figure}[ht]
    \begin{tikzpicture}
        \begin{semilogxaxis}[
            scale=1.5,
            legend pos=south west,
            xlabel=Number of warps, ylabel=Efficiency,  title=Efficiency of the system solution]
        \addplot[color=blue,mark=*] coordinates {
        (3456, 0.025)
        (1728, 0.050)
        (864, 0.100)
        (432, 0.163)
        (216, 0.258)
        (108, 0.293)
        (54, 0.312)
        (27, 0.356)    
        };
        \addlegendentry{$7$ levels (22k)}
        \addplot[color=orange,mark=*] coordinates {
        (3456, 0.089)
        (1728, 0.161)
        (864, 0.260)
        (432, 0.299)
        (216, 0.327)
        (108, 0.370)
        (54, 0.382)
        (27, 0.425)
        };
        \addlegendentry{$8$ levels (88k)}
        \addplot[color=brown,mark=*] coordinates {
        (3456, 0.153)
        (1728, 0.269)
        (864, 0.329)
        (432, 0.375)
        (216, 0.389)
        (108, 0.417)
        (54, 0.434)
        (27, 0.465)
        };
        \addlegendentry{$9$ levels (350k)}
        \addplot[color=violet,mark=*] coordinates {
        (3456, 0.173)
        (1728, 0.332)
        (864, 0.401)
        (432, 0.422)
        (216, 0.432)
        (108, 0.454)
        (54, 0.465)
        (27, 0.491)
        };
        \addlegendentry{$10$ levels (1.4M)}
    
        \end{semilogxaxis}
    
    \end{tikzpicture}
    \caption{Efficiency of the system solution with different data sizes}
    \label{fig:system_efficiency}
\end{figure}
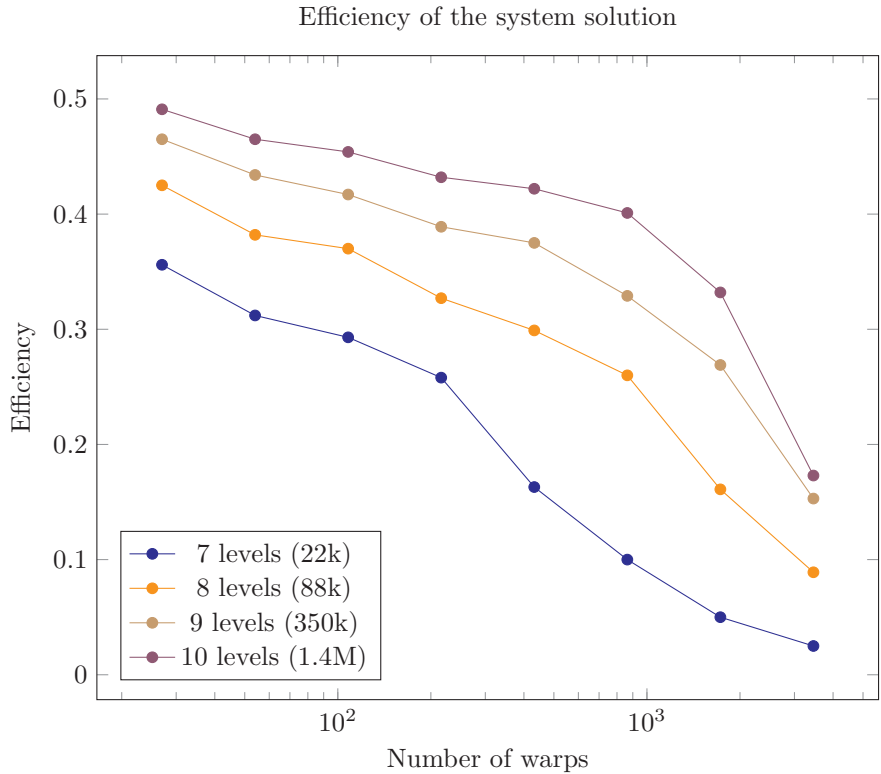

As expected, the smaller the number of processors involved the higher the efficiency. As the number of processors increase we have a regular decrease of the efficiency, that starts collapsing concurrently to the saturation of performance enhancement. Indeed, looking the 7 levels graphs, for $216$ warps, the system's graphical resolution undergoes a substantial shift, plateauing and concomitantly exhibiting a pronounced decline in efficiency. The same phenomena occurs for $864$ for 8 levels and at $1728$ for 9/10 levels. 
On the other direction, is emphasize how increasing the size of the problem for a given parallel settings leads to a higher efficiency.

Result confirms that the algorithm is weakly scalable, due to its complexity, even it can be seen how increasing the number of points lead to trajectories more close to constant efficiency.

\section*{Acknowledgement}
We would like to acknowledge very helpful discussions with 
Lorenz Gollwitzer and Holger Wendland. 
Moreover, we would like to thank the IT-support  of the Fachbereich 12 of Marburg university, in particular Markus M\"uhling, for technical advice about the GPU cluster.   
This work is funded by the Deutsche Forschungsgemeinschaft (DFG, German Research Foundation) –Projektnummer 452806809.

\end{document}